\documentclass{amsart}
\usepackage[dvips,final]{graphics}
\usepackage{arydshln}
\usepackage[makeroom]{cancel}
 \usepackage[all]{xy}
 \usepackage{url}
\usepackage{multirow, blkarray}
\usepackage{booktabs}
\usepackage{textcomp}
 \usepackage[final]{epsfig}
 \usepackage{color}
\usepackage[T1]{fontenc}      
\usepackage[english,french]{babel}
\usepackage[utf8]{inputenc}
\usepackage{blindtext}

\usepackage{amsfonts,amscd,array, mathdots, epigraph}
\usepackage{amsmath}
\usepackage{amssymb}
\usepackage{amsthm}
\usepackage{mathrsfs}
\usepackage{stmaryrd}

\begin{document}


\newtheorem{thm}{Théorème}[section]
\newtheorem{theo}{Théorème}
\newtheorem{defn}[thm]{Définition}
\newtheorem{prop}[thm]{Proposition}
\newtheorem{cor}[thm]{Corollaire}
\newtheorem{con}{Conjecture}
\newtheorem*{rem}{Remarque}
\newtheorem*{rems}{Remarques}
\newtheorem{pro}{Problème}
\newtheorem*{ex}{Exemples}
\newtheorem*{exe}{Exemple}
\newtheorem{lem}[thm]{Lemme}


\title{Combinatoire des sous-groupes de congruence du groupe modulaire}

\author{Flavien Mabilat}

\date{}


\address{
Flavien Mabilat,
Laboratoire de Mathématiques 
U.F.R. Sciences Exactes et Naturelles 
Moulin de la Housse - BP 1039 
51687 Reims cedex 2,
France}
\email{flavien.mabilat@univ-reims.fr}

\maketitle

\selectlanguage{english}
\begin{abstract}
In this paper, we study the combinatorics of congruence subgroups of the modular group by generalizing results obtained in the non-modular case. For this, we define a notion of irreducible solutions from which we can build all the solutions. In particular, we give a particular solution, irreducible for any $N$, and the list of irreducible solutions for $N \leq 6$.
\end{abstract}

\selectlanguage{french}
\begin{abstract}
Dans cet article, on étudie la combinatoire des sous-groupes de congruence du groupe modulaire en généralisant des résultats obtenus dans le cas non modulaire. On définit pour cela une notion de solutions irréductibles à partir desquelles on peut construire l'ensemble des solutions. En particulier, on donne une solution particulière, irréductible pour $N$ quelconque, et la description explicite des solutions irréductibles pour $N \leq 6$.
\\
\end{abstract}

\thispagestyle{empty}

{\bf Mots clés:} groupe modulaire, sous-groupe de congruence, quiddité
\\
\begin{flushright}
\og \textit{Puisque ces mystères nous dépassent, feignons d'en être l'organisateur.} \fg
\\ Jean Cocteau, \textit{Les Mariés de la tour Eiffel}
\end{flushright}

\section{Introduction}

La connaissance de parties génératrices à deux éléments du groupe modulaire 
$$SL_{2}(\mathbb{Z})=
\left\{
\begin{pmatrix}
a & b \\
c & d
   \end{pmatrix}
 \;\vert\;a,b,c,d \in \mathbb{Z},\;
 ad-bc=1
\right\}$$ est l'une des propriétés les plus remarquables de ce groupe. On peut notamment choisir les deux éléments suivants (voir par exemple \cite{A}): \[T=\begin{pmatrix}
 1 & 1 \\[2pt]
    0    & 1 
   \end{pmatrix}, S=\begin{pmatrix}
   0 & -1 \\[2pt]
    1    & 0 
   \end{pmatrix}.
 \] 
\\On déduit de ce choix que pour tout élément $A$ de $SL_{2}(\mathbb{Z})$ il existe un entier strictement positif $n$ et des entiers strictement positifs $a_{1},\ldots,a_{n}$ tels que \[A=T^{a_{n}}ST^{a_{n-1}}S\cdots T^{a_{1}}S=\begin{pmatrix}
   a_{n} & -1 \\[4pt]
    1    & 0 
   \end{pmatrix}
\begin{pmatrix}
   a_{n-1} & -1 \\[4pt]
    1    & 0 
   \end{pmatrix}
   \cdots
   \begin{pmatrix}
   a_{1} & -1 \\[4pt]
    1    & 0 
    \end{pmatrix}.\]
On utilisera la notation $M_{n}(a_{1},\ldots,a_{n})$ pour désigner la matrice $\begin{pmatrix}
   a_{n} & -1 \\[4pt]
    1    & 0 
   \end{pmatrix}
\begin{pmatrix}
   a_{n-1} & -1 \\[4pt]
    1    & 0 
   \end{pmatrix}
   \cdots
   \begin{pmatrix}
   a_{1} & -1 \\[4pt]
    1    & 0 
    \end{pmatrix}$. Remarquons que l'écriture d'un élément de $SL_{2}(\mathbb{Z})$ sous cette forme n'est pas unique (pour une façon d'assurer l'unicité d'une écriture de cette forme on peut consulter \cite{MO}).
\\
\\L'écriture des éléments du groupe modulaire sous cette forme et l'absence d'unicité incite, d'une part, à  essayer de trouver toutes les écritures d'une matrice sous la forme $M_{n}(c_{1},\ldots,c_{n})$ et, d'autre part, à chercher des descriptions combinatoires des solutions en utilisant l'idée générale que des entiers strictement positifs comptent des objets (notamment géométriques). V.Ovsienko (voir \cite{O}) a donné  notamment les solutions (et une description combinatoire de celles-ci en terme de découpages de polygones) de l'équation suivante: \begin{equation}
M_{n}(a_1,\ldots,a_n)=\pm Id.
\end{equation} En particulier, ce résultat généralise un théorème antérieur dû à Conway et Coxeter (voir \cite{CoCo,Cox,Hen}). On dispose également de résultats analogues sur l'équation suivante (voir \cite{Ma}): \begin{equation}
M_{n}(a_1,\ldots,a_n)=\pm S.
\end{equation}
Une autre façon d'exploiter l'écriture des éléments du groupe modulaire sous la forme $M_{n}(a_{1},\ldots,a_{n})$ est de chercher toutes les écritures des éléments d'un sous-groupe donné. Notre objectif ici est de mener à bien cette démarche dans le cas des sous-groupes de congruence suivants: \[\hat{\Gamma}(N)=\{A \in SL_{2}(\mathbb{Z})~{\rm tel~que}~A= \pm Id~( {\rm mod}~N)\}.\] Ce problème est équivalent à la résolution de l'équation suivante dans $\mathbb{Z}/N \mathbb{Z}$ :
\begin{equation}
\tag{$E_{N}$}
M_{n}(a_1,\ldots,a_n)=\pm Id.
\end{equation} Notons que l'équation $(E_{N})$ apparaît naturellement dans la théorie des frises de Coxeter (voir \cite{Mo1,Mo2}). Les solutions de $(E_{N})$ étant invariantes par permutations circulaires on considère les solutions $(a_{1},\ldots,a_{n})$ comme des séquences infinies $n$-périodiques. On dispose déjà des solutions dans le cas où $N=2$ (voir \cite{M} et la section 4). Pour mener à bien la résolution de cette équation, on définit une notion de solution irréductible à partir de laquelle on pourra construire l'ensemble des solutions (voir section suivante). On s'intéressera en particulier à la résolution de $(E_{N})$ pour les petites valeurs de $N$ (voir section 4) et à la recherche de solutions irréductibles dans le cas général (voir section 3).

\section{Résultats principaux}\label{PDSec}

L'objectif de cette section est de définir la notion d'irréductibilité évoquée dans la section précédente et d'énoncer les résultats principaux de ce texte. Cette notion d'irréductibilité repose sur la notion de somme  introduite dans \cite{C} (voir aussi \cite{WZ}). Sauf mention contraire, $N$ désigne un entier naturel supérieur à $2$ et si $a \in \mathbb{Z}$ on note $\overline{a}:=a+N\mathbb{Z}$.

\begin{defn}[\cite{WZ}, définition 1.8]

Soient $(\overline{a_{1}},\ldots,\overline{a_{n}}) \in (\mathbb{Z}/N \mathbb{Z})^{n}$ et $(\overline{b_{1}},\ldots,\overline{b_{m}}) \in (\mathbb{Z}/N \mathbb{Z})^{m}$. On définit l'opération suivante: \[(\overline{a_{1}},\ldots,\overline{a_{n}}) \oplus (\overline{b_{1}},\ldots,\overline{b_{m}}):= (\overline{a_{1}+b_{m}},\overline{a_{2}},\ldots,\overline{a_{n-1}},\overline{a_{n}+b_{1}},\overline{b_{2}},\ldots,\overline{b_{m-1}}).\] Le $(n+m-2)$-uplet obtenu est appelé la somme de $(\overline{a_{1}},\ldots,\overline{a_{n}})$ avec $(\overline{b_{1}},\ldots,\overline{b_{m}})$.

\end{defn}

\begin{ex}

{\rm Voici quelques exemples de somme:
\begin{itemize}
\item $(\overline{1},\overline{2},\overline{1}) \oplus (\overline{2},\overline{0},\overline{1},\overline{2})= (\overline{3},\overline{2},\overline{3},\overline{0},\overline{1})$,
\item $(\overline{3},\overline{2},\overline{1},\overline{1}) \oplus (\overline{1},\overline{0},\overline{1}) = (\overline{4},\overline{2},\overline{1},\overline{2},\overline{0})$,
\item $n \geq 2$, $(\overline{a_{1}},\ldots,\overline{a_{n}}) \oplus (\overline{0},\overline{0}) = (\overline{0},\overline{0}) \oplus (\overline{a_{1}},\ldots,\overline{a_{n}})=(\overline{a_{1}},\ldots,\overline{a_{n}})$.
\end{itemize}
}
\end{ex}

\begin{rem}

{\rm L'opération ci-dessus n'est pas commutative. En effet, on a dans \cite{WZ} l'exemple suivant:}
\[(\overline{1},\overline{1},\overline{1}) \oplus (\overline{2},\overline{1},\overline{2},\overline{1})=(\overline{2},\overline{1},\overline{3},\overline{1},\overline{2}) \neq (\overline{3},\overline{1},\overline{2},\overline{2},\overline{1}) = (\overline{2},\overline{1},\overline{2},\overline{1}) \oplus (\overline{1},\overline{1},\overline{1}),~(N \neq 1).\]

\end{rem}

On montre en particulier que la somme de deux solutions est encore une solution (voir \cite{C} lemme 2.7 et section suivante). Avant de définir la notion de solution irréductible on a encore besoin de la définition suivante:

\begin{defn}[\cite{WZ}, définition 1.5]

Soient $(\overline{a_{1}},\ldots,\overline{a_{n}}) \in (\mathbb{Z}/N \mathbb{Z})^{n}$ et $(\overline{b_{1}},\ldots,\overline{b_{n}}) \in (\mathbb{Z}/N \mathbb{Z})^{n}$. On dit que $(\overline{a_{1}},\ldots,\overline{a_{n}}) \sim (\overline{b_{1}},\ldots,\overline{b_{n}})$ si $(\overline{b_{1}},\ldots,\overline{b_{n}})$ est obtenu par permutation circulaire de $(\overline{a_{1}},\ldots,\overline{a_{n}})$ ou de $(\overline{a_{n}},\ldots,\overline{a_{1}})$.
\end{defn}

On peut montrer (\cite{WZ}, lemme 1.7) que $\sim$ est une relation d'équivalence. En particulier, si on a $(\overline{a_{1}},\ldots,\overline{a_{n}}) \sim (\overline{b_{1}},\ldots,\overline{b_{n}})$ alors $(\overline{a_{1}},\ldots,\overline{a_{n}})$ est solution de $(E_{N})$ si et seulement si $(\overline{b_{1}},\ldots,\overline{b_{n}})$ est solution de $(E_{N})$ (voir \cite{C} proposition 2.6 et également la section 3). On peut maintenant définir la notion d'irréductibilité annoncée dans l'introduction.

\begin{defn}[\cite{C}, définition 2.9]

Une solution $(\overline{c_{1}},\ldots,\overline{c_{n}})$ avec $n \geq 3$ de $(E_{N})$ est dite réductible s'il existe deux solutions de $(E_{N})$ $(\overline{a_{1}},\ldots,\overline{a_{m}})$ et $(\overline{b_{1}},\ldots,\overline{b_{l}})$ telles que \begin{itemize}
\item $(\overline{c_{1}},\ldots,\overline{c_{n}}) \sim (\overline{a_{1}},\ldots,\overline{a_{m}}) \oplus (\overline{b_{1}},\ldots,\overline{b_{l}})$,
\item $m \geq 3$ et $l \geq 3$.
\end{itemize}
Une solution est dite irréductible si elle n'est pas réductible.
\end{defn}

\begin{rem} 

{\rm $(\overline{0},\overline{0})$ n'est pas considérée comme étant une solution irréductible de $(E_{N})$.}

\end{rem}

\noindent Cette notion d'irréductibilité est celle que l'on va utiliser  pour résoudre $(E_{N})$. 
\\
\\Pour $N=1$, l'équation $(E_{N})$ n'a pas d'intérêt et, pour $N=0$, l'équation se ramène à la résolution dans $\mathbb{Z}$ de l'équation $M_{n}(a_{1},\ldots,a_{n})=\pm Id$. On dispose dans ce cas du résultat suivant:

\begin{thm}[Cuntz, Holm; \cite{C}, Théorème 3.2]

L'ensemble des solutions irréductibles de $(E_{0})$ est \[\{(1,1,1),(-1,-1,-1),(a,0,-a,0),(0,-a,0,a), a \in \mathbb{Z}-\{\pm 1\}\}. \]

\end{thm}

\noindent On dispose également d'une description combinatoire de ces solutions (voir \cite{CH} Théorème 7.3).
\\
\\On s'intéresse donc dans cette article aux cas $N \geq 2$. On va établir les résultats suivants:

\begin{thm}

i)Les solutions irréductibles de $(E_{2})$ sont $(\overline{1},\overline{1},\overline{1})$ et $(\overline{0},\overline{0},\overline{0},\overline{0})$.
\\
\\ii)Les solutions irréductibles de $(E_{3})$ sont $(\overline{1},\overline{1},\overline{1})$, $(\overline{-1},\overline{-1},\overline{-1})$ et $(\overline{0},\overline{0},\overline{0},\overline{0})$.
\\
\\iii)Les solutions irréductibles de $(E_{4})$ sont $(\overline{1},\overline{1},\overline{1})$, $(\overline{-1},\overline{-1},\overline{-1})$, $(\overline{0},\overline{0},\overline{0},\overline{0})$, $(\overline{0},\overline{2},\overline{0},\overline{2})$ ,$(\overline{2},\overline{0},\overline{2},\overline{0})$ et $(\overline{2},\overline{2},\overline{2},\overline{2})$.
\\
\\iv)Les solutions irréductibles de $(E_{5})$ sont (à permutations cycliques près) $(\overline{1},\overline{1},\overline{1})$, $(\overline{-1},\overline{-1},\overline{-1})$, $(\overline{0},\overline{0},\overline{0},\overline{0})$, $(\overline{0},\overline{2},\overline{0},\overline{3})$, $(\overline{2},\overline{2},\overline{2},\overline{2},\overline{2})$, $(\overline{3},\overline{3},\overline{3},\overline{3},\overline{3})$, $(\overline{3},\overline{2},\overline{2},\overline{3},\overline{2},\overline{2})$, $(\overline{2},\overline{3},\overline{3},\overline{2},\overline{3},\overline{3})$, $(\overline{2},\overline{3},\overline{2},\overline{3},\overline{2},\overline{3})$.
\\
\\v)Les solutions irréductibles de $(E_{6})$ sont (à permutations cycliques près) $(\overline{1},\overline{1},\overline{1})$, $(\overline{-1},\overline{-1},\overline{-1})$, $(\overline{0},\overline{0},\overline{0},\overline{0})$, $(\overline{2},\overline{4},\overline{2},\overline{4})$, $(\overline{2},\overline{3},\overline{4},\overline{3})$, $(\overline{0},\overline{2},\overline{0},\overline{4})$, $(\overline{0},\overline{3},\overline{0},\overline{3})$, $(\overline{2},\overline{2},\overline{2},\overline{2},\overline{2},\overline{2})$, $(\overline{3},\overline{3},\overline{3},\overline{3},\overline{3},\overline{3})$, $(\overline{4},\overline{4},\overline{4},\overline{4},\overline{4},\overline{4})$.

\end{thm}

Ces résultats sont démontrés dans la section 4. On donne également en section 4 les solutions irréductibles de $(E_{7})$ obtenues avec assistance informatique.
\\
\\On montre également dans la section 3 le résultat suivant:

\begin{thm}

Si $N \geq 3$, $(\overline{2},\ldots,\overline{2}) \in (\mathbb{Z}/N\mathbb{Z})^{N}$ est une solution irréductible de $(E_{N})$.

\end{thm}
									
\section{Résultats généraux sur l'équation $(E_{N})$}

Dans cette partie, $N$ est un entier naturel supérieur ou égal à 2. On dira qu'une solution de $(E_{N})$ est de taille $n$ si cette solution est un $n$-uplet d'éléments de $\mathbb{Z}/N\mathbb{Z}$ solution de $(E_{N})$.

\subsection{Solutions de $(E_{N})$ pour les petites valeurs de $n$}

\leavevmode\par
\leavevmode\par \noindent On va essayer dans cette sous-partie de rechercher les solutions de $M_{n}(\overline{a_{1}},\ldots,\overline{a_{n}})=\pm Id$ pour les petites valeurs de $n$. On voit facilement que $(E_{N})$ n'a pas de solution pour $n=1$. On va maintenant résoudre $(E_{N})$ pour $n=2$ et $n=3$.

\begin{prop}

$(\overline{0},\overline{0})$ est la seule solution de $(E_{N})$ de taille 2.

\end{prop}

\begin{proof}

$\begin{pmatrix}
    \overline{a_{2}} & \overline{-1} \\
    \overline{1}    & \overline{0} 
    \end{pmatrix} \begin{pmatrix}
    \overline{a_{1}} & \overline{-1} \\
    \overline{1}    & \overline{0} 
    \end{pmatrix}=\begin{pmatrix}
    \overline{a_{2}a_{1}-1} & \overline{-a_{2}} \\
    \overline{a_{1}}    & \overline{-1} 
    \end{pmatrix}$.

\noindent Si $(\overline{a_{1}},\overline{a_{2}})$ est solution de $(E_{N})$ alors $\overline{a_{1}}=\overline{a_{2}}=\overline{0}$ et $(\overline{0},\overline{0})$ est solution de $(E_{N})$.

\end{proof}

\begin{prop}

$(\overline{1},\overline{1},\overline{1})$ et $(\overline{-1},\overline{-1},\overline{-1})$ sont les seules solutions de $(E_{N})$ de taille 3.

\end{prop}

\begin{proof}

\begin{eqnarray*}
\begin{pmatrix}
    \overline{a_{3}} & \overline{-1} \\
    \overline{1}    & \overline{0} 
    \end{pmatrix} \begin{pmatrix}
    \overline{a_{2}} & \overline{-1} \\
    \overline{1}    & \overline{0} 
    \end{pmatrix} \begin{pmatrix}
    \overline{a_{1}} & \overline{-1} \\
    \overline{1}    & \overline{0} 
    \end{pmatrix} &=& \begin{pmatrix}
    \overline{a_{3}a_{2}-1} & \overline{-a_{3}} \\
    \overline{a_{2}}    & \overline{-1} 
    \end{pmatrix}\begin{pmatrix}
    \overline{a_{1}} & \overline{-1} \\
    \overline{1}    & \overline{0} 
    \end{pmatrix} \\
		              &=& \begin{pmatrix}
    \overline{a_{3}a_{2}a_{1}-a_{3}-a_{1}} & \overline{-a_{3}a_{2}+1} \\
    \overline{a_{2}a_{1}-1}    & \overline{-a_{2}} 
    \end{pmatrix}. \\
\end{eqnarray*}

 Si $(\overline{a_{1}},\overline{a_{2}},\overline{a_{3}})$ est solution de $(E_{N})$ alors soit $\overline{a_{2}}=\overline{1}$ et dans ce cas $\overline{a_{1}}=\overline{1}$ et $\overline{a_{3}}=\overline{1}$, soit $\overline{a_{2}}=\overline{-1}$ et dans ce cas $\overline{a_{1}}=\overline{-1}$ et $\overline{a_{3}}=\overline{-1}$. On vérifie que $(\overline{1},\overline{1},\overline{1})$ et $(\overline{-1},\overline{-1},\overline{-1})$ sont solutions de $(E_{N})$.

\end{proof}

\noindent Pour $n=4$, on dispose du résultat suivant:

\begin{prop}

Les solutions de $(E_{N})$ pour $n=4$ sont les 4-uplets suivants $(\overline{-a},\overline{b},\overline{a},\overline{-b})$ avec $\overline{a}\overline{b}=\overline{0}$ et $(\overline{a},\overline{b},\overline{a},\overline{b})$ avec $\overline{a}\overline{b}=\overline{2}$.

\end{prop}

\begin{proof}

\begin{eqnarray*}
M_{4}(\overline{a_{1}},\overline{a_{2}},\overline{a_{3}},\overline{a_{4}}) &=& \begin{pmatrix}
    \overline{a_{4}a_{3}-1} & \overline{-a_{4}} \\
    \overline{a_{3}}    & \overline{-1} 
    \end{pmatrix}\begin{pmatrix}
    \overline{a_{2}} & \overline{-1} \\
    \overline{1}    & \overline{0} 
    \end{pmatrix}\begin{pmatrix}
    \overline{a_{1}} & \overline{-1} \\
    \overline{1}    & \overline{0} 
    \end{pmatrix}\\ 
		               &=& \begin{pmatrix}
    \overline{a_{4}a_{3}a_{2}-a_{4}-a_{2}} & \overline{-a_{4}a_{3}+1} \\
    \overline{a_{3}a_{2}-1}    & \overline{-a_{3}} 
    \end{pmatrix}\begin{pmatrix}
    \overline{a_{1}} & \overline{-1} \\
    \overline{1}    & \overline{0} 
    \end{pmatrix}\\
		              &=&\begin{pmatrix}
    \overline{a_{4}a_{3}a_{2}a_{1}-a_{4}a_{1}-a_{2}a_{1}-a_{4}a_{3}+1} & \overline{-a_{4}a_{3}a_{2}+a_{4}+a_{2}} \\
    \overline{a_{3}a_{2}a_{1}-a_{1}-a_{3}}    & \overline{1-a_{3}a_{2}} 
    \end{pmatrix}.\\
\end{eqnarray*}

Si $(\overline{a_{1}},\overline{a_{2}},\overline{a_{3}},\overline{a_{4}})$ est solution de $(E_{N})$ alors on a deux possibilités:

\begin{itemize}

\item $\overline{a_{3}}\overline{a_{2}}=\overline{0}$ et dans ce cas on a $\overline{-a_{4}a_{3}a_{2}}+\overline{a_{4}}+\overline{a_{2}}=\overline{a_{4}}+\overline{a_{2}}=\overline{0}$ et donc $\overline{a_{4}}=\overline{-a_{2}}$ et on a également $\overline{a_{3}}\overline{a_{2}}\overline{a_{1}}-\overline{a_{1}}-\overline{a_{3}}=-\overline{a_{1}}-\overline{a_{3}}=\overline{0}$ et donc $\overline{a_{1}}=-\overline{a_{3}}$.
\\

\item $\overline{a_{3}}\overline{a_{2}}=\overline{2}$ et dans ce cas on a $\overline{-a_{4}a_{3}a_{2}}+\overline{a_{4}}+\overline{a_{2}}=-\overline{a_{4}}+\overline{a_{2}}=\overline{0}$ et donc $\overline{a_{4}}=\overline{a_{2}}$ et on a également $\overline{a_{3}}\overline{a_{2}a_{1}}-\overline{a_{1}}-\overline{a_{3}}=\overline{a_{1}}-\overline{a_{3}}=\overline{0}$ et donc $\overline{a_{1}}=\overline{a_{3}}.$
\\

\end{itemize}

On vérifie en faisant le calcul que $(\overline{-a},\overline{b},\overline{a},\overline{-b})$ avec $\overline{a}\overline{b}=\overline{0}$ et $(\overline{a},\overline{b},\overline{a},\overline{b})$ avec $\overline{a}\overline{b}=\overline{2}$ sont solutions.

\end{proof}

\begin{rem}
{\rm Pour $n=4$, les solutions dépendent de la structure de $\mathbb{Z}/N \mathbb{Z}$. Par exemple, si $N$ est premier alors $\overline{a_{3}}\overline{a_{2}}=\overline{0}$ implique $\overline{a_{2}}=\overline{0}$ ou $\overline{a_{3}}=\overline{0}$ mais si $N=4$ alors on a par exemple $\overline{2} \times \overline{2}=\overline{0}$.}
\end{rem}

\begin{prop}

i)Les solutions de $(E_{N})$ de taille 3 sont irréductibles.
\\ii)Une solutions de $(E_{N})$ de taille 4 réductible contient $\overline{1}$ ou $\overline{-1}$.

\end{prop}

\begin{proof}

i)Si un 3-uplet est somme d'un $m$-uplet avec un $l$-uplet alors $3=m+l-2$ et donc $m+l=5$ ce qui implique $m \leq 2$ ou $l \leq 2$. Donc, les solutions de $(E_{N})$ de taille 3 sont irréductibles.
\\
\\ii)Soit $(\overline{a_{1}},\overline{a_{2}},\overline{a_{3}},\overline{a_{4}})$ une solution de $(E_{N})$. 
\\
\\Si $(\overline{a_{1}},\overline{a_{2}},\overline{a_{3}},\overline{a_{4}})$ est réductible alors $(\overline{a_{1}},\overline{a_{2}},\overline{a_{3}},\overline{a_{4}})$ est équivalent à la somme d'un $m$-uplet solution de $(E_{N})$ avec un $l$-uplet solution de $(E_{N})$ avec $m, l \geq 3$. On a $m+l-2=4$ donc $m+l=6$ et comme $m, l \geq 3$ on a nécessairement $m=l=3$. Comme un 3-uplet solution de $(E_{N})$ contient $\overline{1}$ ou $\overline{-1}$, une solution réductible de $(E_{N})$ de taille 4 contient $\overline{1}$ ou $\overline{-1}$.

\end{proof}

\begin{rem}
{ \rm On démontre dans la proposition 3.8 la réciproque de ii).}
\end{rem}

\subsection{Opérations sur les solutions}

\leavevmode\par
\leavevmode\par L'objectif de cette partie est de justifier un certain nombre d'assertions présentes dans la section précédente et de donner des façons de construire des solutions à partir de solutions connues. La plupart de ces résultats ont déjà été démontrés dans \cite{CH} et \cite{WZ} mais on les redémontre ici afin d'avoir une présentation complète.

\begin{prop}

Si $(\overline{a_{1}},\ldots,\overline{a_{n}})$ et $(\overline{b_{1}},\ldots,\overline{b_{m}})$ sont solutions de $(E_{N})$ alors $(\overline{a_{1}},\ldots,\overline{a_{n}},\overline{b_{1}},\ldots,\overline{b_{m}})$ est solution de $(E_{N})$. En particulier, $(\overline{a_{1}},\ldots,\overline{a_{n}},\overline{a_{1}},\ldots,\overline{a_{n}})$ est solution de $(E_{N})$.

\end{prop}

\begin{proof}

$\exists (\epsilon,\mu) \in \{\pm 1\}$ tels que $M_{n}(\overline{a_{1}},\ldots,\overline{a_{n}})=\overline{\epsilon}Id$ et $M_{m}(\overline{b_{1}},\ldots,\overline{b_{m}})=\overline{\mu}Id$ (car $(\overline{a_{1}},\ldots,\overline{a_{n}})$ et $(\overline{b_{1}},\ldots,\overline{b_{m}})$ sont solutions de $(E_{N})$). Donc, on a
\[M_{n+m}(\overline{a_{1}},\ldots,\overline{a_{n}},\overline{b_{1}},\ldots,\overline{b_{m}})=M_{m}(\overline{b_{1}},\ldots,\overline{b_{m}})M_{n}(\overline{a_{1}},\ldots,\overline{a_{n}})=\overline{\epsilon\mu}Id.\] Donc, comme $\epsilon\mu  \in \{\pm 1\}$, $(\overline{a_{1}},\ldots,\overline{a_{n}},\overline{b_{1}},\ldots,\overline{b_{m}})$ est solution de $(E_{N})$.

\end{proof}

\begin{prop}

i) $(\overline{a_{1}},\ldots,\overline{a_{n}})$ est solution de $(E_{N})$ si et seulement si $(\overline{a_{n}},\ldots,\overline{a_{1}})$ est solution de $(E_{N})$.
\\
\\ii) $(\overline{a_{1}},\ldots,\overline{a_{n}})$ est solution de $(E_{N})$ si et seulement si $(\overline{-a_{1}},\ldots,\overline{-a_{n}})$ est solution de $(E_{N})$.
\\
\\iii) Si $(\overline{a_{1}},\ldots,\overline{a_{n}}) \sim (\overline{b_{1}},\ldots,\overline{b_{n}})$ alors $(\overline{a_{1}},\ldots,\overline{a_{n}})$ est solution de $(E_{N})$ si et seulement si $(\overline{b_{1}},\ldots,\overline{b_{n}})$ est solution de $(E_{N})$.

\end{prop}

\begin{proof} La preuve suivante est une adaptation de la remarque 2.6 de \cite{CH}.
\\
\\i) On pose $K=\begin{pmatrix}
    0 & 1 \\
    1    & 0 
    \end{pmatrix}$.
\\
\\Supposons $(\overline{a_{1}},\ldots,\overline{a_{n}})$ solution de $(E_{N})$. On a  $\begin{pmatrix}
   \overline{a} & \overline{-1} \\
    \overline{1}    & \overline{0} 
    \end{pmatrix}^{-1}=\overline{K}\begin{pmatrix}
   \overline{a} & \overline{-1} \\
    \overline{1}    & \overline{0} 
    \end{pmatrix}\overline{K}$ et $\overline{K}^{2}=Id$. Donc,
		
\begin{eqnarray*}
M_{n}(\overline{a_{n}},\ldots,\overline{a_{1}}) &=& (\overline{K}\begin{pmatrix}
    \overline{a_{1}} & \overline{-1} \\
    \overline{1}    & \overline{0} 
    \end{pmatrix}^{-1}\overline{K})\ldots(\overline{K}\begin{pmatrix}
    \overline{a_{n}} & \overline{-1} \\
    \overline{1}    & \overline{0} 
    \end{pmatrix}^{-1}\overline{K})\\ 
		                   &=& \overline{K}(\begin{pmatrix}
    \overline{a_{n}} & \overline{-1} \\
    \overline{1}    & \overline{0} 
    \end{pmatrix}\ldots\begin{pmatrix}
    \overline{a_{1}} & \overline{-1} \\
    \overline{1}    & \overline{0} 
    \end{pmatrix})^{-1}\overline{K}\\
		              &=& \overline{K}M_{n}(\overline{a_{1}},\ldots,\overline{a_{n}})^{-1}\overline{K}\\
									&=& \pm Id~{\rm car}~(\overline{a_{1}},\ldots,\overline{a_{n}})~{\rm solution~de}~(E_{N}).\\								
\end{eqnarray*}

\noindent Donc, $(\overline{a_{n}},\ldots,\overline{a_{1}})$ est solution de $(E_{N})$. 
\\
\\Si $(\overline{a_{n}},\ldots,\overline{a_{1}})$ est solution de $(E_{N})$ alors par ce qui précède $(\overline{a_{1}},\ldots,\overline{a_{n}})$ est solution de $(E_{N})$.
\\
\\ii) Si $A \in SL_{2}(\mathbb{Z}/N\mathbb{Z})$ on note $A^{T}$ la transposée de $A$.
On a, \begin{eqnarray*}
M_{n}(\overline{-a_{1}},\ldots,\overline{-a_{n}}) &=& \begin{pmatrix}
    \overline{-a_{n}} & \overline{-1} \\
    \overline{1}    & \overline{0} 
    \end{pmatrix}\ldots\begin{pmatrix}
    \overline{-a_{1}} & \overline{-1} \\
    \overline{1}    & \overline{0} 
    \end{pmatrix}\\ 
		                   &=& (\overline{-1})^{n}\begin{pmatrix}
    \overline{a_{n}} & \overline{1} \\
    \overline{-1}    & \overline{0} 
    \end{pmatrix}\ldots\begin{pmatrix}
    \overline{a_{1}} & \overline{1} \\
    \overline{-1}    & \overline{0} 
    \end{pmatrix}\\
		              &=& (\overline{-1})^{n}\begin{pmatrix}
   \overline{a_{n}} & \overline{-1} \\
    \overline{1}    & \overline{0} 
   \end{pmatrix}^{T}\ldots\begin{pmatrix}
   \overline{a_{1}} & \overline{-1} \\
    \overline{1}    & \overline{0} 
    \end{pmatrix}^{T}\\
									&=& (\overline{-1})^{n}(\begin{pmatrix}
   \overline{a_{1}} & \overline{-1} \\
    \overline{1}    & \overline{0} 
   \end{pmatrix}\ldots\begin{pmatrix}
   \overline{a_{n}} & \overline{-1} \\
    \overline{1}    & \overline{0} 
    \end{pmatrix})^{T}\\
									&=&(\overline{-1})^{n}M_{n}(\overline{a_{n}},\ldots,\overline{a_{1}})^{T}.
\end{eqnarray*}

Donc, $(\overline{-a_{1}},\ldots,\overline{-a_{n}})$ est solution de $(E_{N})$ si et seulement si $(\overline{a_{n}},\ldots,\overline{a_{1}})$ est solution de $(E_{N})$. Par i), $(\overline{-a_{1}},\ldots,\overline{-a_{n}})$ est solution de $(E_{N})$ si et seulement si $(\overline{a_{1}},\ldots,\overline{a_{n}})$ est solution de $(E_{N})$.
\\
\\ iii) C'est une conséquence de i) et de l'invariance par permutations circulaires des solutions.
		
\end{proof}

\begin{prop}

Soit $(\overline{b_{1}},\ldots,\overline{b_{m}})$ une solution de $(E_{N})$. Soit $(\overline{a_{1}},\ldots,\overline{a_{n}}) \in (\mathbb{Z}/N\mathbb{Z})^{n}$ alors la somme $(\overline{a_{1}},\ldots,\overline{a_{n}}) \oplus (\overline{b_{1}},\ldots,\overline{b_{m}})$ est solution de $(E_{N})$ si et seulement si $(\overline{a_{1}},\ldots,\overline{a_{n}})$ est solution de $(E_{N})$.

\end{prop}

\begin{proof}

La preuve suivante est adaptée de la preuve du lemme 2.7 de \cite{C} et de la preuve du lemme 1.9 de \cite{WZ}.
\\
\\$(\overline{b_{1}},\ldots,\overline{b_{m}})$ est une solution de $(E_{N})$ donc il existe $\mu$ appartenant à $\{\pm 1\}$ tel que $M_{m}(\overline{b_{1}},\ldots,\overline{b_{m}})=\overline{\mu}Id$.
\\
\\Si $(\overline{a_{1}},\ldots,\overline{a_{n}})$ est solution de $(E_{N})$. $\exists \epsilon \in \{\pm 1\}$ tel que $M_{n}(\overline{a_{1}},\ldots,\overline{a_{n}})=\overline{\epsilon}Id$. On vérifie que \[M_{1}(\overline{a_{n}+b_{1}})=\overline{-1} \times M_{1}(\overline{b_{1}})M_{1}(\overline{0})M_{1}(\overline{a_{n}})=\overline{-1} \times M_{1}(\overline{a_{n}})M_{1}(\overline{0})M_{1}(\overline{b_{1}}),\] \[M_{1}(\overline{a_{1}+b_{m}})=\overline{-1} \times M_{1}(\overline{a_{1}})M_{1}(\overline{0})M_{1}(\overline{b_{m}})=\overline{-1} \times M_{1}(\overline{b_{m}})M_{1}(\overline{0})M_{1}(\overline{a_{1}}).\] 
\\Notons $M=M_{n+m-2}(\overline{a_{1}+b_{m}},\overline{a_{2}},\ldots,\overline{a_{n-1}},\overline{a_{n}+b_{1}},\overline{b_{2}},\ldots,\overline{b_{m-1}})$. On a
\begin{eqnarray*}
M &=& M_{m-2}(\overline{b_{2}},\ldots,\overline{b_{m-1}})M_{1}(\overline{a_{n}+b_{1}})M_{n-2}(\overline{a_{2}},\ldots,\overline{a_{n-1}})M_{1}(\overline{a_{1}+b_{m}}) \\
  &=& M_{m-2}(\overline{b_{2}},\ldots,\overline{b_{m-1}})M_{1}(\overline{b_{1}})M_{1}(\overline{0})M_{1}(\overline{a_{n}})M_{n-2}(\overline{a_{2}},\ldots,\overline{a_{n-1}})M_{1}(\overline{a_{1}})M_{1}(\overline{0})M_{1}(\overline{b_{m}}) \\
	&=& M_{m-1}(\overline{b_{1}},\ldots,\overline{b_{m-1}})M_{1}(\overline{0})M_{n}(\overline{a_{1}},\ldots,\overline{a_{n}})M_{1}(\overline{0})M_{1}(\overline{b_{m}}) \\
	&=& M_{m-1}(\overline{b_{1}},\ldots,\overline{b_{m-1}})M_{1}(\overline{0})(\overline{\epsilon}Id)M_{1}(\overline{0})M_{1}(\overline{b_{m}}) \\
	&=& \overline{\epsilon} M_{m-1}(\overline{b_{1}},\ldots,\overline{b_{m-1}})M_{2}(\overline{0},\overline{0})M_{1}(\overline{b_{m}}) \\
	&=& \overline{\epsilon} M_{m-1}(\overline{b_{1}},\ldots,\overline{b_{m-1}})(\overline{-1}Id)M_{1}(\overline{b_{m}}) \\
	&=& \overline{-\epsilon} M_{m}(\overline{b_{m}},\overline{b_{1}},\ldots,\overline{b_{m-1}}) \\
	&=& \overline{-\mu\epsilon} Id. \\
\end{eqnarray*}

\noindent Donc, $(\overline{a_{1}},\ldots,\overline{a_{n}}) \oplus (\overline{b_{1}},\ldots,\overline{b_{m}})$ est solution de $(E_{N})$.
\\
\\Si $(\overline{a_{1}},\ldots,\overline{a_{n}}) \oplus (\overline{b_{1}},\ldots,\overline{b_{m}})$ est solution de $(E_{N})$. $(\overline{a_{1}+b_{m}},\overline{a_{2}},\ldots,\overline{a_{n-1}},\overline{a_{n}+b_{1}},\overline{b_{2}},\ldots,\overline{b_{m-1}})$ est solution de $(E_{N})$ donc $(\overline{a_{2}},\ldots,\overline{a_{n-1}},\overline{a_{n}+b_{1}},\overline{b_{2}},\ldots,\overline{b_{m-1}},\overline{a_{1}+b_{m}})$ est solution de $(E_{N})$. $\exists \alpha \in \{\pm 1\}$ tel que $M_{n+m-2}(\overline{a_{2}},\ldots,\overline{a_{n-1}},\overline{a_{n}+b_{1}},\overline{b_{2}},\ldots,\overline{b_{m-1}},\overline{a_{1}+b_{m}})= \overline{\alpha} Id$. On a
\begin{eqnarray*}
\overline{\alpha} Id &=& M_{1}(\overline{a_{1}+b_{m}})M_{m-2}(\overline{b_{2}},\ldots,\overline{b_{m-1}})M_{1}(\overline{a_{n}+b_{1}})M_{n-2}(\overline{a_{2}},\ldots,\overline{a_{n-1}}) \\
  &=& M_{1}(\overline{a_{1}})M_{1}(\overline{0})M_{1}(\overline{b_{m}})M_{m-2}(\overline{b_{2}},\ldots,\overline{b_{m-1}})M_{1}(\overline{b_{1}})M_{1}(\overline{0})M_{1}(\overline{a_{n}})M_{n-2}(\overline{a_{2}},\ldots,\overline{a_{n-1}}) \\
	&=& M_{1}(\overline{a_{1}})M_{1}(\overline{0})M_{m}(\overline{b_{1}},\ldots,\overline{b_{m}})M_{1}(\overline{0})M_{n-1}(\overline{a_{2}},\ldots,\overline{a_{n}}) \\
	&=&  \overline{\mu}M_{1}(\overline{a_{1}})M_{2}(\overline{0},\overline{0})M_{n-1}(\overline{a_{2}},\ldots,\overline{a_{n}})\\
	&=&  \overline{-\mu}M_{n}(\overline{a_{2}},\ldots,\overline{a_{n}},\overline{a_{1}}).\\
\end{eqnarray*} 

\noindent Ainsi, $M_{n}(\overline{a_{2}},\ldots,\overline{a_{n}},\overline{a_{1}})=\overline{-\alpha\mu}Id$. Donc, $(\overline{a_{2}},\ldots,\overline{a_{n}},\overline{a_{1}})$ est solution de $(E_{N})$ et donc $(\overline{a_{1}},\ldots,\overline{a_{n}})$ est solution de $(E_{N})$.

\end{proof}

On déduit de ce résultat qu'une solution $(\overline{c_{1}},\ldots,\overline{c_{n}})$ avec $n \geq 3$ de $(E_{N})$ est  réductible s'il existe une solution de $(E_{N})$ $(\overline{b_{1}},\ldots,\overline{b_{l}})$ et un $m$-uplet $(\overline{a_{1}},\ldots,\overline{a_{m}})$ tels que $m \geq 3$ et $l \geq 3$ et \[(\overline{c_{1}},\ldots,\overline{c_{n}}) \sim (\overline{a_{1}},\ldots,\overline{a_{m}}) \oplus (\overline{b_{1}},\ldots,\overline{b_{l}}).\]

On en déduit le résultat suivant

\begin{prop}

i) Si $n \geq 4$ alors une solution de $(E_{N})$ contenant $\overline{1}$ ou $\overline{-1}$ est réductible.
\\ ii) Si $n \geq 5$ alors une solution de $(E_{N})$ contenant $\overline{0}$ est réductible.

\end{prop}

\begin{proof}

i) Soit $(\overline{a_{1}},\ldots,\overline{a_{n}})$ (avec $n \geq 4$) une solution de $(E_{N})$. Si $\exists \epsilon \in \{\pm \overline{1}\}$ et si $\exists i \in [\![1;n]\!]$ tels que $\overline{a_{i}}=\overline{\epsilon}$ alors on a \[(\overline{a_{i+1}}\ldots,\overline{a_{n}},\overline{a_{1}},\ldots,\overline{a_{i}})=(\overline{a_{i+1}-\epsilon}\ldots,\overline{a_{n}},\overline{a_{1}},\ldots,\overline{a_{i-1}-\epsilon}) \oplus (\overline{\epsilon},\overline{\epsilon},\overline{\epsilon}).\]
Donc, $(\overline{a_{1}},\ldots,\overline{a_{n}})$ est réductible.
\\
\\ii) Soit $(\overline{a_{1}},\ldots,\overline{a_{n}})$ (avec $n \geq 5$) une solution de $(E_{N})$. Si $\exists i \in [\![1;n]\!]$ tel que $\overline{a_{i}}=\overline{0}$ alors on a \[(\overline{a_{i+2}}\ldots,\overline{a_{n}},\overline{a_{1}},\ldots,\overline{a_{i}},\overline{a_{i+1}})=(\overline{a_{i+2}}\ldots,\overline{a_{n}},\overline{a_{1}},\ldots,\overline{a_{i-1}+a_{i+1}}) \oplus (\overline{-a_{i+1}},\overline{0},\overline{a_{i+1}},\overline{0}).\]
Donc, $(\overline{a_{1}},\ldots,\overline{a_{n}})$ est réductible.

\end{proof}

\begin{rem}

{\rm La réciproque est fausse. Par exemple, si $N=4$, $(\overline{2},\overline{2},\overline{2},\overline{2},\overline{2},\overline{2},\overline{2},\overline{2})$ ne contient pas $\overline{1}$, $\overline{-1}$ ou $\overline{0}$ mais est une solution de $(E_{4})$ réductible puisque $(\overline{2},\overline{2},\overline{2},\overline{2},\overline{2},\overline{2},\overline{2},\overline{2})=(\overline{0},\overline{2},\overline{2},\overline{2},\overline{2},\overline{0}) \oplus (\overline{2},\overline{2},\overline{2},\overline{2})$.}

\end{rem}

\subsection{Solutions monomiales minimales}

\leavevmode\par
\leavevmode\par Dans cette sous-partie, on s'intéresse aux solutions dont toutes les composantes sont identiques. En particulier, on cherche à connaitre des solutions valables pour tout $N$ (ou au moins pour des valeurs de $N$ vérifiant certaines propriétés) et à savoir si elles sont ou non irréductibles.

\subsubsection{Définitions et premiers résultats} 

\leavevmode\par
\leavevmode\par On commence par la définition suivante: 

\begin{defn}

Soient $n \in \mathbb{N}^{*}$ et $\overline{k} \in \mathbb{Z}/N\mathbb{Z}$. On appelle solution $(n,\overline{k})$-monomiale un $n$-uplet d'\'el\'ements de $\mathbb{Z}/ N \mathbb{Z}$ constitu\'e uniquement de $\overline{k} \in \mathbb{Z}/N\mathbb{Z}$ et solution de $(E_{N})$.
\\
\\ On appelle solution monomiale une solution pour laquelle il existe $m \in \mathbb{N}^{*}$ et $\overline{l} \in \mathbb{Z}/N\mathbb{Z}$ tels qu'elle est $(m,\overline{l})$-monomiale.
\\
\\ On appelle solution $\overline{k}$-monomiale minimale une solution $(n,\overline{k})$-monomiale avec $n$ le plus petit entier pour lequel il existe une solution $(n,\overline{k})$-monomiale.
\\
\\ On appelle solution monomiale minimale une solution $\overline{k}$-monomiale minimale pour un $\overline{k} \in \mathbb{Z}/N\mathbb{Z}$.

\end{defn}

\begin{ex}

{\rm $(\overline{1},\overline{1},\overline{1},\overline{1},\overline{1},\overline{1})$ est une solution $(6,\overline{1})$-monomiale de $(E_{N})$ et $(\overline{1},\overline{1},\overline{1})$ est une solution $\overline{1}$-monomiale minimale de $(E_{N})$. 
}

\end{ex}

\begin{rems}
{\rm i) Si $(\overline{k},\ldots,\overline{k}) \in (\mathbb{Z}/N\mathbb{Z})^{n}$ est une solution $\overline{k}$-monomiale minimale de $(E_{N})$ alors, par la proposition 3.6 ii), $(\overline{-k},\ldots,\overline{-k}) \in (\mathbb{Z}/N\mathbb{Z})^{n}$ est une solution $\overline{-k}$-monomiale minimale de $(E_{N})$.
\\
\\ ii) La taille d'une solution $\overline{k}$-monomiale minimale de $(E_{N})$ est l'ordre de $\begin{pmatrix}
   \overline{k}   & \overline{-1} \\
   \overline{1} & \overline{0}
\end{pmatrix}$ dans le groupe $PSL_{2}(\mathbb{Z}/N\mathbb{Z})$. En particulier, si $N$ est premier alors la taille d'une solution $\overline{k}$-monomiale minimale de $(E_{N})$ est inférieure à $N$ car $N$ est l'ordre maximal des éléments de $PSL_{2}(\mathbb{Z}/N\mathbb{Z})$ (voir \cite{KS}). Cela n'est plus vrai si $N$ n'est pas premier. Par exemple, si $N=10$, une solution $\overline{3}$-monomiale minimale de $(E_{10})$ est de taille 15.
}
\end{rems}

On commence par le résultat suivant donnant une solution valable dans le cas où $N$ est un carré.

\begin{prop}

Si $N=l^{2}$ avec $l \geq 2$ alors $(\overline{l},\ldots,\overline{l}) \in (\mathbb{Z}/N\mathbb{Z})^{2l}$ est solution de $(E_{N})$.

\end{prop}

\begin{proof}

On a 
\begin{eqnarray*}
M_{2l}(\overline{l},\ldots,\overline{l}) &=& (\begin{pmatrix}
   \overline{l}   & \overline{-1} \\
   \overline{1} & \overline{0}
\end{pmatrix}\begin{pmatrix}
   \overline{l}   & \overline{-1} \\
   \overline{1} & \overline{0}
\end{pmatrix})^{l} \\
                                         &=& \begin{pmatrix}
   \overline{l^{2}-1}   & \overline{-l} \\
   \overline{l} & \overline{-1}
\end{pmatrix}^{l} \\
                                         &=& \begin{pmatrix}
   \overline{-1}   & \overline{-l} \\
   \overline{l} & \overline{-1}
\end{pmatrix}^{l} \\
                                         &=& \overline{(-Id+lS)^{l}} \\
																				 &=& \overline{ \sum_{k=0}^{l} \binom{l}{k} (-1)^{l-k}(lS)^{k}}~{\rm(bin\hat{o}me~de~Newton)} \\
																				 &=& \overline{ (-1)^{l}\binom{l}{0}Id+(-1)^{l-1}\binom{l}{1}lS+l^{2} \sum_{k=2}^{l} \binom{l}{k} (-1)^{l-k} l^{k-2} S^{k}} \\
																				 &=& \overline{(-1)^{l}Id+(-1)^{l-1}l^{2}S}\\
																				 &=& \overline{(-1)^{l}Id}.\\
\end{eqnarray*}

\end{proof}

\noindent La question de l'irréductibilité de ces solutions est résolue dans la proposition suivante.

\begin{prop}

Soit $N=l^{2}$ avec $l \geq 2$. $(\overline{l},\ldots,\overline{l}) \in (\mathbb{Z}/N\mathbb{Z})^{2l}$ est irréductible si et seulement si $l=2$.

\end{prop}

\begin{proof}

Si $l=2$, alors la solution est $(\overline{2},\overline{2},\overline{2},\overline{2})$ et elle est irréductible puisqu'elle ne contient pas $\pm \overline{1}$. Si $l \geq 3$ alors la solution n'est pas irréductible car \[(\overline{l},\ldots,\overline{l})=(\overline{2l},\overline{l},\ldots,\overline{l},\overline{2l}) \oplus (\overline{-l},\overline{l},\overline{l},\overline{-l})\] et $(\overline{-l},\overline{l},\overline{l},\overline{-l})$ est solution (par la proposition 3.3) et $(\overline{2l},\overline{l},\ldots,\overline{l},\overline{2l})$ est de taille $2l-2 \geq 4$.

\end{proof}

\begin{rem}
{\rm La démonstration précédente montre également que si $N=l^{2}$ avec $l \geq 3$ alors le $(2l-2)$-uplet $(\overline{2l},\overline{l},\ldots,\overline{l},\overline{2l}) \in (\mathbb{Z}/N\mathbb{Z})^{2l-2}$ est solution de $(E_{N})$.
}
\end{rem}

On va maintenant donner une généralisation de la proposition 3.10. Avant cela on a besoin du résultat classique suivant :

\begin{lem}

Soient $n \in \mathbb{N}^{*}$ et $k \in [\![1;n]\!]$, $\frac{n}{{\rm pgcd}(n,k)}~{\rm divise}~\binom{n}{k}$.

\end{lem}

\begin{proof}

\[{n \choose k}=\frac{n}{k}{n-1 \choose k-1}=\frac{\frac{n}{{\rm pgcd}(n,k)}}{\frac{k}{{\rm pgcd}(n,k)}}{n-1 \choose k-1}.\] Donc, comme ${n \choose k} \in \mathbb{N}^{*}$, on a $\frac{k}{{\rm pgcd}(n,k)}~{\rm divise}~\frac{n}{{\rm pgcd}(n,k)}{n-1 \choose k-1}$. Comme $\frac{k}{{\rm pgcd}(n,k)}$ et $\frac{n}{{\rm pgcd}(n,k)}$ sont premiers entre eux, on a, par le lemme de Gauss, $\frac{k}{{\rm pgcd}(n,k)}~{\rm divise}~{n-1 \choose k-1}$. Donc, $\frac{n}{{\rm pgcd}(n,k)}~{\rm divise}~{n \choose k}$.

\end{proof}

Avec ce lemme on peut démontrer le lemme suivant :

\begin{lem}

Soient $n \in \mathbb{N}^{*}$, $n \geq 2$, $l \in \mathbb{N}^{*}$, $l \geq 2$ et $j \in [\![1;n-1]\!]$, on a $l^{n-j}$ divise $\binom{l^{n-1}}{j}$.

\end{lem}

\begin{proof}

Si $j=1$ alors $\binom{l^{n-1}}{j}=l^{n-1}$ et donc le résultat est vrai. Si $j=2$ alors \[\binom{l^{n-1}}{j}=\frac{l^{n-1}(l^{n-1}-1)}{2}.\] Si $l$ est pair on a $\frac{l^{n-1}(l^{n-1}-1)}{2}=l^{n-2}\frac{l}{2}(l^{n-1}-1)$ et si $l$ est impair alors $(l^{n-1}-1)$ est pair et on a $\frac{l^{n-1}(l^{n-1}-1)}{2}=l^{n-1}\frac{l^{n-1}-1}{2}$. Dans tous les cas, $l^{n-2}$ divise $\binom{l^{n-1}}{j}$. On peut donc supposer $n \geq 4$ et $j \geq 3$.
\\
\\ Par le lemme précédent, $\frac{l^{n-1}}{{\rm pgcd}(l^{n-1},j)}~{\rm divise}~\binom{l^{n-1}}{j}$. Notons $l=p_{1}^{\alpha_{1}} \ldots p_{r}^{\alpha_{r}}$ la décomposition de $l$ en facteurs premiers. $\exists(\beta_{1},\ldots,\beta_{r}) \in \mathbb{N}^{r}$ tel que ${\rm pgcd}(l^{n-1},j)=p_{1}^{\beta_{1}} \ldots p_{r}^{\beta_{r}}$. 
\\
\\Montrons que $\forall i \in [\![1;r]\!]$, $\beta_{i} \leq \alpha_{i}(j-1)$. Supposons par l'absurde qu'il existe un entier $i$ dans $[\![1;r]\!]$ tel que $\beta_{i} > \alpha_{i}(j-1)$. Par récurrence, on montre que si $j \geq 3$ on a $p_{i}^{j-1} > j$. On a $p_{i}^{\alpha_{i}(j-1)} \geq p_{i}^{j-1} > j$ et donc ${\rm pgcd}(l^{n-1},j) >j$ ce qui est absurde.
\\
\\Ainsi, $\forall i \in [\![1;r]\!]$, $\beta_{i} \leq \alpha_{i}(j-1)$ et donc $l^{n-j}$ divise $\frac{l^{n-1}}{{\rm pgcd}(l^{n-1},j)}$. On en déduit que $l^{n-j}~{\rm divise}~\binom{l^{n-1}}{j}$.

\end{proof}

\begin{prop}

Si $N=l^{n}$ avec $l \geq 2$ et $n \geq 2$ alors $(\overline{l},\ldots,\overline{l}) \in (\mathbb{Z}/N\mathbb{Z})^{2l^{n-1}}$ est solution de $(E_{N})$. 

\end{prop}

\begin{proof}

On a 
\begin{eqnarray*}
M_{2l^{n-1}}(\overline{l},\ldots,\overline{l}) &=& (\begin{pmatrix}
   \overline{l}   & \overline{-1} \\
   \overline{1} & \overline{0}
\end{pmatrix}\begin{pmatrix}
   \overline{l}   & \overline{-1} \\
   \overline{1} & \overline{0}
\end{pmatrix})^{l^{n-1}} \\
                                         &=& \begin{pmatrix}
   \overline{l^{2}-1}   & \overline{-l} \\
   \overline{l} & \overline{-1}
\end{pmatrix}^{l^{n-1}} \\
                                         &=& \overline{(-Id+l\begin{pmatrix}
   \overline{l}   & \overline{-1} \\
   \overline{1} & \overline{0}
\end{pmatrix})^{l^{n-1}}} \\
																				 &=& \overline{ \sum_{k=0}^{l^{n-1}} \binom{l^{n-1}}{k} l^{k}(-1)^{l^{n-1}-k}\begin{pmatrix}
   \overline{l}   & \overline{-1} \\
   \overline{1} & \overline{0}
\end{pmatrix}^{k}}~{\rm(bin\hat{o}me~de~Newton)} \\
																				 &=& \overline{ \sum_{k=0}^{n-1} \binom{l^{n-1}}{k} l^{k}(-1)^{l^{n-1}-k}\begin{pmatrix}
   \overline{l}   & \overline{-1} \\
   \overline{1} & \overline{0}
\end{pmatrix}^{k}}\\
																				 &=& \overline{ (-1)^{l^{n-1}}Id+\sum_{k=1}^{n-1} \binom{l^{n-1}}{k} l^{k}(-1)^{l^{n-1}-k}\begin{pmatrix}
   \overline{l}   & \overline{-1} \\
   \overline{1} & \overline{0}
\end{pmatrix}^{k}}\\
																				 &=& \overline{(-1)^{l^{n-1}}Id}~{\rm car}~l^{n-k}~{\rm divise}~\binom{l^{n-1}}{k}~{\rm par~le~lemme~3.13}.\\
\end{eqnarray*}

\end{proof}

Notons que d'après le théorème 2.6 (démontré dans la sous-section suivante), la solution ci-dessus est irréductible si $l=2$.
\\
\\On cherche maintenant à étudier l'irréductibilité des solutions monomiales minimales. On a besoin pour cela du résultat suivant sur l'expression de la matrice $M_{n}(a_{1},\ldots,a_{n})$ en terme de déterminant. On pose $K_{-1}=0$, $K_{0}=1$ et on note pour $i \geq 1$ \[K_i(a_{1},\ldots,a_{i})=
\left|
\begin{array}{cccccc}
a_1&1&&&\\[4pt]
1&a_{2}&1&&\\[4pt]
&\ddots&\ddots&\!\!\ddots&\\[4pt]
&&1&a_{i-1}&\!\!\!\!\!1\\[4pt]
&&&\!\!\!\!\!1&\!\!\!\!a_{i}
\end{array}
\right|.\] $K_{i}(a_{1},\ldots,a_{i})$ est le continuant de $a_{1},\ldots,a_{i}$. On dispose de l'égalité suivante (voir \cite{MO,O}) \[M_{n}(a_{1},\ldots,a_{n})=\begin{pmatrix}
    K_{n}(a_{1},\ldots,a_{n}) & -K_{n-1}(a_{2},\ldots,a_{n}) \\
    K_{n-1}(a_{1},\ldots,a_{n-1})  & -K_{n-2}(a_{2},\ldots,a_{n-1}) 
   \end{pmatrix}.\]

\noindent Ceci nous permet d'avoir le résultat préliminaire suivant :

\begin{prop}

Soient $n \in \mathbb{N}^{*}$, $n \geq 3$ et $(\overline{a},\overline{b},\overline{k}) \in (\mathbb{Z}/N\mathbb{Z})^{3}$. 
\\Si $(\overline{a},\overline{k},\overline{k},\ldots,\overline{k},\overline{b}) \in (\mathbb{Z}/N\mathbb{Z})^{n}$ est solution de $(E_{N})$ alors $\overline{a}=\overline{b}$ et on a \[\overline{a}(\overline{a}-\overline{k})=\overline{0}.\]
	
\end{prop}	
	
\begin{proof}

Comme $(\overline{a},\overline{k},\overline{k},\ldots,\overline{k},\overline{b})$ est solution de $(E_{N})$, $\exists \epsilon \in \{-1,1\}$ tel que \[\overline{\epsilon} Id=M_{n}(\overline{a},\overline{k},\overline{k},\ldots,\overline{k},\overline{b})=\begin{pmatrix}
   K_{n}(\overline{a},\overline{k},\ldots,\overline{k},\overline{b})   & -K_{n-1}(\overline{k},\ldots,\overline{k},\overline{b}) \\
   K_{n-1}(\overline{a},\overline{k},\ldots,\overline{k}) & -K_{n-2}(\overline{k},\ldots,\overline{k})
\end{pmatrix}.\] Donc, \[K_{n-1}(\overline{a},\overline{k},\ldots,\overline{k})=-K_{n-1}(\overline{k},\ldots,\overline{k},\overline{b})=\overline{0}~{\rm et}~K_{n-2}(\overline{k},\ldots,\overline{k})=-\overline{\epsilon}.\] Or, $K_{n-1}(\overline{a},\overline{k},\ldots,\overline{k})=\overline{a}K_{n-2}(\overline{k},\ldots,\overline{k})-K_{n-3}(\overline{k},\ldots,\overline{k})=\overline{-\epsilon a}-K_{n-3}(\overline{k},\ldots,\overline{k})$. Ainsi, comme $\overline{\epsilon}^{2}=\overline{1}$, on a \[\overline{a}=\overline{-\epsilon}K_{n-3}(\overline{k},\ldots,\overline{k}).\] De même, $K_{n-1}(\overline{k},\ldots,\overline{k},\overline{b})=\overline{b}K_{n-2}(\overline{k},\ldots,\overline{k})-K_{n-3}(\overline{k},\ldots,\overline{k})=\overline{-\epsilon b}-K_{n-3}(\overline{k},\ldots,\overline{k})$. Il en découle que, \[\overline{b}=\overline{-\epsilon}K_{n-3}(\overline{k},\ldots,\overline{k}).\] Donc, \[\overline{a}=\overline{b}.\] De plus, on a $\overline{-\epsilon}=K_{n-2}(\overline{k},\ldots,\overline{k})=\overline{k}K_{n-3}(\overline{k},\ldots,\overline{k})-K_{n-4}(\overline{k},\ldots,\overline{k})$ et \[M_{n-2}(\overline{k},\overline{k},\ldots,\overline{k})=\begin{pmatrix}
   K_{n-2}(\overline{k},\ldots,\overline{k})   & -K_{n-3}(\overline{k},\ldots,\overline{k}) \\
   K_{n-3}(\overline{k},\ldots,\overline{k}) & -K_{n-4}(\overline{k},\ldots,\overline{k})
\end{pmatrix} \in SL_{2}(\mathbb{Z}/N\mathbb{Z}).\] Ainsi, $-K_{n-2}(\overline{k},\ldots,\overline{k})K_{n-4}(\overline{k},\ldots,\overline{k})+K_{n-3}(\overline{k},\ldots,\overline{k})^{2}=\overline{1}$. Or, comme $K_{n-2}(\overline{k},\ldots,\overline{k})=\overline{-\epsilon}$, on a \[\overline{\epsilon}K_{n-4}(\overline{k},\ldots,\overline{k})+K_{n-3}(\overline{k},\ldots,\overline{k})^{2}=\overline{1}\] c'est-à-dire \[K_{n-4}(\overline{k},\ldots,\overline{k})=\overline{\epsilon}(\overline{1}-K_{n-3}(\overline{k},\ldots,\overline{k})^{2}).\] Donc, on a
\begin{eqnarray*}
\overline{-\epsilon} &=& \overline{k}K_{n-3}(\overline{k},\ldots,\overline{k})-K_{n-4}(\overline{k},\ldots,\overline{k}) \\
                     &=& \overline{k}K_{n-3}(\overline{k},\ldots,\overline{k})-\overline{\epsilon}(\overline{1}-K_{n-3}(\overline{k},\ldots,\overline{k})^{2}) \\ 
										 &=& \overline{k}K_{n-3}(\overline{k},\ldots,\overline{k})-\overline{\epsilon}+\overline{\epsilon}K_{n-3}(\overline{k},\ldots,\overline{k})^{2} \\
										 &=& -\overline{\epsilon k a}-\overline{\epsilon}+\overline{\epsilon a^{2}}. \\
\end{eqnarray*}

\noindent On en déduit, $\overline{0}=-\overline{\epsilon k a}+\overline{\epsilon a^{2}}=\overline{\epsilon}\overline{a}(\overline{a}-\overline{k})$ et donc $\overline{0}=\overline{a}(\overline{a}-\overline{k}).$

\end{proof}

\begin{rem} {\rm Il est possible que $\overline{a} \neq \overline{0}$ et $\overline{a} \neq \overline{k}$. Par exemple, si $N=9$, $(\overline{6},\overline{3},\overline{3},\overline{6})$ est solution de $(E_{9})$.}

\end{rem}

\begin{thm}

Si $N$ est premier alors toute solution monomiale minimale de $(E_{N})$ différente de $(\overline{0},\overline{0})$ est irréductible.

\end{thm}

\begin{proof}

Soient $\overline{k} \in \mathbb{Z}/N\mathbb{Z}$, $\overline{k} \neq \overline{0}$ et $n \in \mathbb{N}^{*}$ tels que $(\overline{k},\ldots,\overline{k}) \in (\mathbb{Z}/N\mathbb{Z})^{n}$ soit monomiale minimale. On suppose par l'absurde que cette solution peut s'écrire comme une somme de deux solutions non triviales.
\\
\\Il existe $(\overline{a_{1}},\ldots,\overline{a_{l}})$ et $(\overline{b_{1}},\ldots,\overline{b_{l'}})$ solutions de $(E_{N})$ différentes de $(\overline{0},\overline{0})$ avec $l+l'=n+2$ et $l,l' \geq 3$ telles que \[(\overline{k},\ldots,\overline{k})=(\overline{b_{1}+a_{l}},\overline{b_{2}},\ldots,\overline{b_{l'-1}},\overline{b_{l'}+a_{1}},\overline{a_{2}},\ldots,\overline{a_{l-1}}).\] On a donc $\overline{a_{2}}=\ldots=\overline{a_{l-1}}=\overline{k}$. Comme $(\overline{a_{1}},\ldots,\overline{a_{l}})$ est solution de $(E_{N})$, on a par la proposition précédente $\overline{a_{1}}=\overline{a_{l}}=\overline{a}$ et $\overline{0}=\overline{a}(\overline{a}-\overline{k})$.
\\
\\Puisque $N$ est premier, $\mathbb{Z}/N\mathbb{Z}$ est intègre et donc l'équation $\overline{0}=\overline{a}(\overline{a}-\overline{k})$ a pour solutions $\overline{a}=\overline{0}$ et $\overline{a}=\overline{k}$. Si $\overline{a}=\overline{0}$ alors \[(\overline{0},\overline{a_{2}},\ldots,\overline{a_{l-1}},\overline{0}) \sim (\overline{a_{2}},\ldots,\overline{a_{l-1}}) \oplus (\overline{0},\overline{0},\overline{0},\overline{0}).\] Par la proposition 3.7,  $(\overline{a_{2}},\ldots,\overline{a_{l-1}})=(\overline{k},\ldots,\overline{k}) \in (\mathbb{Z}/N\mathbb{Z})^{l-2}$  est encore solution de $(E_{N})$ ce qui contredit la minimalité de la solution. 
\\
\\Ainsi, $\overline{a}=\overline{k}$ et par minimalité de la solution on a $l \geq n$ ce qui implique $l' \leq 2$. Donc, $l'=2$ et $(\overline{b_{1}},\ldots,\overline{b_{l'}})=(\overline{0},\overline{0})$ ce qui est absurde.

\end{proof}

\begin{rem} {\rm Si $N$ n'est pas premier alors une solution monomiale minimale n'est pas forcément irréductible. Par exemple, si $N=9$, $(\overline{3},\overline{3},\overline{3},\overline{3},\overline{3},\overline{3})$ est monomiale minimale mais pas irréductible car $(\overline{3},\overline{3},\overline{3},\overline{3},\overline{3},\overline{3})= (\overline{6},\overline{3},\overline{3},\overline{6}) \oplus (\overline{6},\overline{3},\overline{3},\overline{6})$.}

\end{rem}

On peut améliorer la proposition 3.15 pour traiter le cas où $N=pq$ avec $p$ et $q$ deux nombres premiers distincts.

\begin{lem}

Soient $p$ et $q$ deux nombres premiers distincts et $N=pq$. Soient $n \in \mathbb{N}^{*}$, $n \geq 3$ et $(\overline{a},\overline{b}) \in (\mathbb{Z}/N\mathbb{Z})^{2}$. 
\\i)Si $(\overline{a},\overline{p},\overline{p},\ldots,\overline{p},\overline{b}) \in (\mathbb{Z}/N\mathbb{Z})^{n}$ est solution de $(E_{N})$ alors $\overline{a}=\overline{b}$ et $\overline{a} \in \{\overline{0},\overline{p}\}$.
\\
\\ii)Si $(\overline{a},\overline{q},\overline{q},\ldots,\overline{q},\overline{b}) \in (\mathbb{Z}/N\mathbb{Z})^{n}$ est solution de $(E_{N})$ alors $\overline{a}=\overline{b}$ et $\overline{a} \in \{\overline{0},\overline{q}\}$.
	
\end{lem}	

\begin{proof}

Si $(\overline{a},\overline{p},\overline{p},\ldots,\overline{p},\overline{b}) \in (\mathbb{Z}/N\mathbb{Z})^{n}$ est solution de $(E_{N})$. Par la proposition 3.15, $\overline{a}=\overline{b}$ et $\overline{a}(\overline{a}-\overline{p})=\overline{0}$.
\\
\\Supposons par l'absurde que $\overline{a}$ est un élément inversible de $\mathbb{Z}/N\mathbb{Z}$. Dans ce cas, on a \[\overline{a}(\overline{a}-\overline{p})=\overline{0} \Longleftrightarrow (\overline{a}-\overline{p})=\overline{0} \Longleftrightarrow \overline{a}=\overline{p}.\] Or, $\overline{p}$ n'est pas inversible dans $\mathbb{Z}/N\mathbb{Z}$ ce qui est absurde. Donc, $\overline{a}$ n'est pas un élément inversible de $\mathbb{Z}/N\mathbb{Z}$.
\\
\\ Donc, soit il existe un entier $i$ dans $[\![0;q-1]\!]$ tel que $\overline{a}=\overline{ip}$ soit il existe un entier $j$ dans $[\![1;p-1]\!]$ tel que $\overline{a}=\overline{jq}$. 
\\
\\Si $\exists i \in [\![0;q-1]\!]$ tel que $\overline{a}=\overline{ip}$. On a \[\overline{a}(\overline{a}-\overline{p})=\overline{ip}(\overline{ip}-\overline{p})=\overline{ip^{2}}(\overline{i}-\overline{1})=\overline{0}.\] Donc, $pq$ divise $ip^{2}(i-1)$ et, en particulier, $q$ divise $ip^{2}(i-1)$. Comme $q$ et $p^{2}$ sont premiers entre eux, on a, par le lemme de Gauss, $q$ divise $i(i-1)$. Si $i \neq 0$ alors $q$ et $i$ sont premiers entre eux (puisque $i \in [\![0;q-1]\!]$), et donc, par le lemme de Gauss, $q$ divise $(i-1)$ ce qui implique $i=1$. Donc, $i=0$ ou $i=1$.
\\
\\Si $\exists j \in [\![1;p-1]\!]$ tel que $\overline{a}=\overline{jq}$. On a \[\overline{a}(\overline{a}-\overline{p})=\overline{jq}(\overline{jq}-\overline{p})=\overline{0}.\] Donc, $pq$ divise $jq(jq-p)$ et, en particulier, $p$ divise $jq(jq-p)$. Comme $p$ et $q$ sont premiers entre eux, on a, par le lemme de Gauss, $p$ divise $j(jq-p)$. Comme $p$ et $j$ sont premiers entre eux (puisque $j \in [\![1;p-1]\!]$), on a, par le lemme de Gauss, $p$ divise $(jq-p)$ et donc $p$ divise $jq$ ce qui est absurde.
\\
\\Donc, $\overline{a} \in \{\overline{0},\overline{p}\}$. On procède de façon analogue pour ii).

\end{proof}

\begin{prop}

Soient $p$ et $q$ deux nombres premiers distincts et $N=pq$. Toute solution $\overline{k}$-monomiale minimale de $(E_{N})$ avec $\overline{k} \in \{\overline{p},\overline{q}\}$ est irréductible.

\end{prop}

\begin{proof}

Soit $n \in \mathbb{N}^{*}$ tels que $(\overline{p},\ldots,\overline{p}) \in (\mathbb{Z}/N\mathbb{Z})^{n}$ soit monomiale minimale. On suppose par l'absurde que cette solution peut s'écrire comme une somme de deux solutions non triviales.
\\
\\Il existe $(\overline{a_{1}},\ldots,\overline{a_{l}})$ et $(\overline{b_{1}},\ldots,\overline{b_{l'}})$ solutions de $(E_{N})$ différentes de $(\overline{0},\overline{0})$ avec $l+l'=n+2$ et $l,l' \geq 3$ telles que \[(\overline{p},\ldots,\overline{p})=(\overline{b_{1}+a_{l}},\overline{b_{2}},\ldots,\overline{b_{l'-1}},\overline{b_{l'}+a_{1}},\overline{a_{2}},\ldots,\overline{a_{l-1}}).\] On a donc $\overline{a_{2}}=\ldots=\overline{a_{l-1}}=\overline{p}$. Comme $(\overline{a_{1}},\ldots,\overline{a_{l}})$ est solution de $(E_{N})$, on a par le lemme précédent $\overline{a_{1}}=\overline{a_{l}}=\overline{a}$ avec $\overline{a}=\overline{0}$ ou $\overline{a}=\overline{p}$. 
\\
\\Si $\overline{a}=\overline{0}$ alors \[(\overline{0},\overline{a_{2}},\ldots,\overline{a_{l-1}},\overline{0}) \sim (\overline{a_{2}},\ldots,\overline{a_{l-1}}) \oplus (\overline{0},\overline{0},\overline{0},\overline{0}).\] Par la proposition 3.7,  $(\overline{a_{2}},\ldots,\overline{a_{l-1}})=(\overline{p},\ldots,\overline{p}) \in (\mathbb{Z}/N\mathbb{Z})^{l-2}$  est encore solution de $(E_{N})$ ce qui contredit la minimalité de la solution. 
\\
\\Donc, $\overline{a}=\overline{p}$ et par minimalité de la solution on a $l \geq n$ ce qui implique $l' \leq 2$. Donc, $l'=2$ et $(\overline{b_{1}},\ldots,\overline{b_{l'}})=(\overline{0},\overline{0})$ ce qui est absurde.
\\
\\On procède de la même façon dans le cas d'une solution $\overline{q}$-monomiale minimale.

\end{proof}

\begin{rem} {\rm Si $N=pq$ et $\overline{k} \notin \{\overline{p},\overline{q}\}$ alors une solution $\overline{k}$-monomiale minimale de $(E_{N})$ n'est pas forcément irréductible. Par exemple, si $N=10=2 \times 5$, une solution $\overline{3}$-monomiale minimale (qui est de taille 15) n'est pas irréductible car on peut l'écrire comme une somme à l'aide de la solution $(\overline{8},\overline{3},\overline{3},\overline{3},\overline{8})$.
}

\end{rem}

\subsubsection{Démonstration du théorème 2.6}

\leavevmode\par
\leavevmode\par \noindent Dans le cas des solutions $\overline{2}$-monomiales on peut améliorer les résultats précédents.

\begin{lem}

Soit $n \in \mathbb{N}^{*}$ alors $M_{n}(2,\ldots,2)=\begin{pmatrix}
   n+1   & -n \\
   n & -n+1
\end{pmatrix}$.

\end{lem}

\begin{proof}

On raisonne par récurrence sur $n$.
\\
\\Si $n=1$ alors le résultat est vrai. On suppose qu'il existe $n \in \mathbb{N}^{*}$ tel que $M_{n}(2,\ldots,2)=\begin{pmatrix}
   n+1   & -n \\
   n & -n+1
\end{pmatrix}$. On a 

\begin{eqnarray*}
M_{n+1}(2,\ldots,2) &=& M_{n}(2,\ldots,2) \begin{pmatrix}
   2   & -1 \\
   1 & 0
\end{pmatrix}\\
                    &=&\begin{pmatrix}
   n+1   & -n \\
   n & -n+1
\end{pmatrix}\begin{pmatrix}
   2   & -1 \\
   1 & 0
\end{pmatrix}\\
                    &=&\begin{pmatrix}
   2n+2-n   & -n-1 \\
   2n-n+1 & -n
\end{pmatrix}\\
                    &=&\begin{pmatrix}
   (n+1)+1   & -(n+1) \\
    n+1 &  -(n+1)+1
\end{pmatrix}.\\
\end{eqnarray*} La formule est vraie pour $n+1$ et donc par récurrence elle est vraie pour tout $n$.

\end{proof}

De ce calcul, on déduit l'existence d'une solution particulière pour $N$ quelconque.

\begin{cor}

$(\overline{2},\ldots,\overline{2}) \in (\mathbb{Z}/N\mathbb{Z})^{N}$ est solution de $(E_{N})$.

\end{cor}

\begin{proof}

$M_{N}(\overline{2},\ldots,\overline{2})=\begin{pmatrix}
   \overline{N+1}   & \overline{-N} \\
   \overline{N} & \overline{-N+1}
\end{pmatrix}=Id.$

\end{proof}

Pour montrer l'irréductibilité de cette solution, on va utiliser une version améliorée de la proposition 3.15 utilisant l'hypothèse $\overline{k}=\overline{2}$.

\begin{lem}

Soient $n \in \mathbb{N}^{*}$, $n \geq 3$ et $(\overline{a},\overline{b}) \in (\mathbb{Z}/N\mathbb{Z})^{2}$. 
\\$(\overline{a},\overline{2},\overline{2},\ldots,\overline{2},\overline{b}) \in (\mathbb{Z}/N\mathbb{Z})^{n}$ est solution de $(E_{N})$ si et seulement si $\overline{a}=\overline{b}=\overline{2}$ et $n \equiv 0 [N]$ ou $\overline{a}=\overline{b}=\overline{0}$ et $n \equiv 2 [N]$.

\end{lem}

\begin{proof}

Supposons que $(\overline{a},\overline{2},\overline{2},\ldots,\overline{2},\overline{b})$ est solution de $(E_{N})$.
\\
\\Par la proposition 3.15, $\overline{a}=\overline{b}$. On a 

\begin{eqnarray*}
M_{n}(\overline{a},\overline{2},\ldots,\overline{2},\overline{a}) &=& \begin{pmatrix}
   \overline{a}   & \overline{-1} \\
    \overline{1} &  \overline{0}
\end{pmatrix}M_{n-2}(\overline{2},\ldots,\overline{2}) \begin{pmatrix}
   \overline{a}   & \overline{-1} \\
    \overline{1} &  \overline{0}
\end{pmatrix} \\
                                                                  &=& \begin{pmatrix}
   \overline{a}   & \overline{-1} \\
    \overline{1} &  \overline{0}
\end{pmatrix}\begin{pmatrix}
   \overline{n-1}   & \overline{-n+2} \\
   \overline{n-2} & \overline{-n+3}
\end{pmatrix} \begin{pmatrix}
   \overline{a}   & \overline{-1} \\
    \overline{1} &  \overline{0}
\end{pmatrix} \\
                                                                  &=& \begin{pmatrix}
   \overline{an-a-n+2}   & \overline{-an+2a+n-3} \\
   \overline{n-1} & \overline{-n+2}
\end{pmatrix} \begin{pmatrix}
   \overline{a}   & \overline{-1} \\
    \overline{1} &  \overline{0}
\end{pmatrix} \\
                                                                  &=& \begin{pmatrix}
   \overline{x}   & \overline{y} \\
    \overline{an-a-n+2} &  \overline{1-n}
\end{pmatrix}. \\
\end{eqnarray*}  

On a deux cas: 
\begin{itemize}
\item $\overline{1-n}=\overline{1}$. Dans ce cas, $\overline{n}=\overline{0}$ c'est-à-dire $n \equiv 0 [N]$ et $\overline{0}=\overline{an-a-n+2}=\overline{-a+2}$ et donc $\overline{a}=\overline{2}$.
\item $\overline{1-n}=\overline{-1}$. Dans ce cas, $\overline{n}=\overline{2}$ c'est-à-dire $n \equiv 2 [N]$ et $\overline{0}=\overline{an-a-n+2}=\overline{a}$.
\\
\end{itemize}

Supposons que $\overline{a}=\overline{b}=\overline{2}$ et $n \equiv 0 [N]$ ou $\overline{a}=\overline{b}=\overline{0}$ et $n \equiv 2 [N]$. D'après le corollaire précédent, $(\overline{2},\ldots,\overline{2}) \in (\mathbb{Z}/N\mathbb{Z})^{N}$ est solution de $(E_{N})$ et $(\overline{0},\overline{0}) \in (\mathbb{Z}/N\mathbb{Z})^{2}$ est solution de $(E_{N})$. Donc, $(\overline{a},\overline{2},\overline{2},\ldots,\overline{2},\overline{b})$ est solution de $(E_{N})$ (proposition 3.5).

\end{proof} 

Ce théorème montre en particulier que $(\overline{2},\ldots,\overline{2}) \in (\mathbb{Z}/N\mathbb{Z})^{N}$ est une solution monomiale minimale de $(E_{N})$.
\\
\\On peut maintenant démontrer le théorème 2.6.

\begin{proof}[Démonstration du théorème 2.6.]

Si $N=3$ alors le résultat est vrai (proposition 3.4) et on suppose maintenant $N \geq 4$. On suppose par l'absurde que cette solution peut s'écrire comme une somme de deux solutions non triviales.
\\
\\Il existe $(\overline{a_{1}},\ldots,\overline{a_{l}})$ et $(\overline{b_{1}},\ldots,\overline{b_{l'}})$ solutions de $(E_{N})$ différentes de $(\overline{0},\overline{0})$ avec $l+l'=N+2$ et $l,l' \geq 3$ telles que \[(\overline{2},\ldots,\overline{2})=(\overline{b_{1}+a_{l}},\overline{b_{2}},\ldots,\overline{b_{l'-1}},\overline{b_{l'}+a_{1}},\overline{a_{2}},\ldots,\overline{a_{l-1}}).\] On a donc $\overline{a_{2}}=\ldots=\overline{a_{l-1}}=\overline{2}$. Comme $(\overline{a_{1}},\ldots,\overline{a_{l}})$ est solution de $(E_{N})$, on a par le lemme précédent $l \equiv 2 [N]$ ou $l \equiv 0 [N]$. Comme $l \geq 3$ on a nécessairement $l \geq N$ et donc $l' \leq 2$. Donc, $l'=2$ et $(\overline{b_{1}},\ldots,\overline{b_{l'}})=(\overline{0},\overline{0})$ ce qui est absurde.
 
\end{proof}

On en déduit le corollaire suivant qui donne une autre solution irréductible dans le cas général.

\begin{cor}

Si $N \geq 3$, $(\overline{N-2},\ldots,\overline{N-2}) \in (\mathbb{Z}/N\mathbb{Z})^{N}$ est une solution irréductible de $(E_{N})$.

\end{cor}

\begin{proof}

Par la proposition 3.6, $(\overline{N-2},\ldots,\overline{N-2}) \in (\mathbb{Z}/N\mathbb{Z})^{N}$ est une solution de $(E_{N})$. On s'intéresse maintenant à l'irréductibilité de la solution.
\\
\\Si $N=3$ alors le résultat est vrai (proposition 3.4) et on suppose maintenant $N \geq 4$. On suppose par l'absurde que cette solution peut s'écrire comme une somme de deux solutions non triviales.
\\
\\Il existe $(\overline{a_{1}},\ldots,\overline{a_{l}})$ et $(\overline{b_{1}},\ldots,\overline{b_{l'}})$ solutions de $(E_{N})$ différentes de $(\overline{0},\overline{0})$ avec $l+l'=N+2$ et $l,l' \geq 3$ telles que \[(\overline{N-2},\ldots,\overline{N-2})=(\overline{b_{1}+a_{l}},\overline{b_{2}},\ldots,\overline{b_{l'-1}},\overline{b_{l'}+a_{1}},\overline{a_{2}},\ldots,\overline{a_{l-1}}).\] On a donc $\overline{a_{2}}=\ldots=\overline{a_{l-1}}=\overline{N-2}$.
\\
\\ De plus, $(\overline{-a_{1}},\ldots,\overline{-a_{l}})$ est solution de $(E_{N})$ et $\overline{-a_{2}}=\ldots=\overline{-a_{l-1}}=\overline{2}$. Donc par le lemme 3.21, $l \equiv 2 [N]$ ou $l \equiv 0 [N]$. Comme $l \geq 3$, on a nécessairement $l \geq N$ et donc $l' \leq 2$. Donc, $l'=2$ et $(\overline{b_{1}},\ldots,\overline{b_{l'}})=(\overline{0},\overline{0})$ ce qui est absurde.
 
\end{proof}

\section{Solution de $(E_{N})$ pour $N \in [\![2;7]\!]$}

\subsection{Cas où $N=2$}

\leavevmode\par
\leavevmode\par On commence par le cas $N=2$ étudié dans \cite{M}.

\begin{thm}[voir \cite{M}, Proposition 5.3]

Les solutions irréductibles de $(E_{2})$ sont $(\overline{1},\overline{1},\overline{1})$ et $(\overline{0},\overline{0},\overline{0},\overline{0})$.

\end{thm}

Ce cas possède une description combinatoire particulièrement élégante nécessitant la définition suivante:

\begin{defn}

i) (\cite{M}, Définition 3.1) On appelle décomposition de type (3|4) le découpage d'un polygone convexe $P$ à $n$ sommets par des diagonales ne se coupant qu'aux sommets et tel que les sous-polygones soient des triangles ou des quadrilatères. 
\\
\\ii) (\cite{M}, Définition 3.3) À chaque sommet de $P$ on associe un élément $\overline{c}$ de $\mathbb{Z}/2\mathbb{Z}$ de la façon suivante $$\overline{c}=
\left\{
\begin{array}{ll}\overline{1}, & \hbox{si le nombre de triangles utilisant ce sommet est impair};\\[2pt]
\overline{0}, & \hbox{si le nombre de triangles utilisant ce sommet est pair}.
\end{array}
\right.
$$
On parcourt les sommets, à partir de n'importe lequel d'entre eux, dans le sens horaire ou le sens trigonométrique, pour obtenir le $n$-uplet $(\overline{c_{1}},\ldots,\overline{c_{n}})$. Ce $n$-uplet est la quiddité de la décomposition de type (3|4) de $P$.

\end{defn}

\begin{rem}
{\rm Si $(\overline{c_{1}},\ldots,\overline{c_{n}})$ est la quiddité d'une décomposition de type (3|4) de $P$ alors tout $n$-uplet équivalent à $(\overline{c_{1}},\ldots,\overline{c_{n}})$ est aussi la quiddité de cette décomposition de $P$.
}
\end{rem}

\begin{thm}[\cite{M}, Théorème 1]
Soit $n \geq 2$.
\\ i) Une solution de $(E_{2})$ de taille $n$ est la quiddité d'une décomposition de type (3|4) d'un polygone convexe à $n$ sommets.
\\ ii) La quiddité d'une décomposition de type (3|4) d'un polygone convexe à $n$ sommets est une solution de $(E_{2})$ de taille $n$.

\end{thm}

\begin{ex}
{\rm
Voici quelques exemples de décomposition de type (3|4) avec leur quiddité:
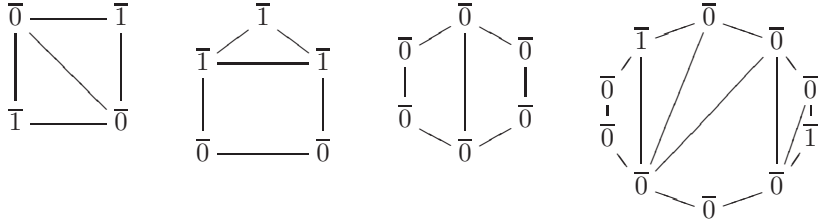
$$
\shorthandoff{; :!?}
\xymatrix @!0 @R=0.7cm @C=0.7cm
{
\overline{0}\ar@{-}[dd]\ar@{-}[rr]\ar@{-}[dd]\ar@{-}[rrdd]&&\overline{1}\ar@{-}[dd]
\\
\\
\overline{1}\ar@{-}[rr]&& \overline{0}
}
\qquad
\xymatrix @!0 @R=0.6cm @C=0.8cm
{
&\overline{1}\ar@{-}[ld]\ar@{-}[rd]&
\\
\overline{1}\ar@{-}[dd]\ar@{-}[rr]&&\overline{1}\ar@{-}[dd]
\\
\\
\overline{0}\ar@{-}[rr]&& \overline{0}
}
\qquad
\xymatrix @!0 @R=0.45cm @C=0.8cm
{
&\overline{0}\ar@{-}[ld]\ar@{-}[rd]\ar@{-}[dddd]&
\\
\overline{0}\ar@{-}[dd]&&\overline{0}\ar@{-}[dd]
\\
\\
\overline{0}&& \overline{0}
\\
&\overline{0}\ar@{-}[lu]\ar@{-}[ru]&
}
\qquad
\xymatrix @!0 @R=0.32cm @C=0.45cm
 {
&&&\overline{0}\ar@{-}[llddddddd]\ar@{-}[lld]\ar@{-}[rrd]&
\\
&\overline{1}\ar@{-}[ldd]\ar@{-}[dddddd]&&&& \overline{0}\ar@{-}[rdd]\ar@{-}[lllldddddd]\ar@{-}[dddddd]\\
\\
\overline{0}\ar@{-}[dd]&&&&&&\overline{0}\ar@{-}[dd]\ar@{-}[ldddd]\\
\\
\overline{0}\ar@{-}[rdd]&&&&&&\overline{1}\ar@{-}[ldd]\\
\\
&\overline{0}\ar@{-}[rrd]&&&& \overline{0}\ar@{-}[lld]\\
&&&\overline{0}&
}
$$
}
\end{ex}

\begin{rem}
{\rm On peut améliorer le théorème précédent. En effet, si $(\overline{c_{1}},\ldots,\overline{c_{n}})$ est une solution de $(E_{2})$ et s'il existe un entier $i$ dans $[\![1;n]\!]$ tel que $\overline{c_{i}} \neq \overline{0}$ alors $(\overline{c_{1}},\ldots,\overline{c_{n}})$ est la quiddité d'une décomposition de type (3|4) d'un polygone convexe à $n$ sommets ne contenant que des triangles (voir \cite{M}, Remarque 5.4). 
}
\end{rem}

\subsection{Cas $N=3$}

\leavevmode\par
\leavevmode\par \noindent On passe maintenant au cas $N=3$.

\subsubsection{Démonstration du théorème 2.5 ii)}

\begin{proof}

Par la proposition 3.4, $(\overline{1},\overline{1},\overline{1})$, $(\overline{-1},\overline{-1},\overline{-1})$ et $(\overline{0},\overline{0},\overline{0},\overline{0})$ sont irréductibles. Par les propositions 3.2, 3.3, 3.4 et 3.8, il n'y a pas d'autres solutions irréductibles pour $n=3, 4$. Soient $n \geq 5$ et $(\overline{a_{1}},\ldots,\overline{a_{n}})$ une solution de $(E_{3})$. $(\overline{a_{1}},\ldots,\overline{a_{n}})$ contient $\overline{0}$, $\overline{1}$ ou $\overline{-1}$, donc, par la proposition 3.8, $(\overline{a_{1}},\ldots,\overline{a_{n}})$ est réductible.

\end{proof}

\noindent On en déduit le résultat suivant:

\begin{prop}

Si $(\overline{a_{1}},\ldots,\overline{a_{n}})$ est solution de $(E_{3})$ alors $\overline{a_{1}}+\ldots+\overline{a_{n}}=\overline{0}$.

\end{prop}

\begin{proof}

On raisonne par récurrence sur $n$. 
\\
\\Le résultat est vrai pour $n=2$, $n=3$ et $n=4$. 
\\
\\On suppose maintenant $n \geq 5$. Soit $(\overline{a_{1}},\ldots,\overline{a_{n}})$ une solution de $(E_{3})$. $(\overline{a_{1}},\ldots,\overline{a_{n}})$ est équivalent à la somme d'un $k$-uplet $(\overline{b_{1}},\ldots,\overline{b_{k}})$ solution de $(E_{3})$ ($k=n-1$ ou $k=n-2$) avec une des solutions irréductibles de $(E_{3})$, $(\overline{c_{1}},\ldots,\overline{c_{l}})$ ($l=3$ ou $l=4$). 
\\
\\$\overline{b_{1}}+\ldots+\overline{b_{k}}=\overline{0}$ (par hypothèse de récurrence) et $\overline{c_{1}}+\ldots+\overline{c_{l}}=\overline{0}$ (par l'initialisation). On a, \[\overline{a_{1}}+\ldots+\overline{a_{n}}=\overline{b_{1}}+\ldots+\overline{b_{k}}+\overline{c_{1}}+\ldots+\overline{c_{l}}=\overline{0}.\] Donc, si $(\overline{a_{1}},\ldots,\overline{a_{n}})$ est solution de $(E_{3})$ alors $\overline{a_{1}}+\ldots+\overline{a_{n}}=\overline{0}$.

\end{proof}

\subsubsection{Description combinatoire des solutions}

\begin{defn}

i) On appelle décomposition pondérée de type (3|4) de première espèce le découpage d'un polygone convexe $P$ à $n$ sommets par des diagonales ne se coupant qu'aux sommets et tel que les sous-polygones soient des triangles de poids $\overline{1}$ ou $\overline{-1}$ ou des quadrilatères de poids $\overline{0}$. 
\\
\\ii) On choisit un sommet de $P$ que l'on numérote par 1 puis on numérote les autres sommets de $P$ en suivant le sens horaire ou le sens trigonométrique. La quiddité de la décomposition pondérée de type (3|4) de première espèce de $P$ est le $n$-uplet $(\overline{c_{1}},\ldots,\overline{c_{n}})$ avec $\overline{c_{i}}$ la somme des poids des sous-polygones utilisant le sommet $i$.

\end{defn}

\begin{rem}
{\rm Si $(\overline{c_{1}},\ldots,\overline{c_{n}})$ est la quiddité de la décomposition pondérée de type (3|4) de première espèce de $P$ alors tout $n$-uplet équivalent à $(\overline{c_{1}},\ldots,\overline{c_{n}})$ est aussi la quiddité de cette décomposition de $P$.
}
\end{rem}

\begin{ex}
{\rm
Voici quelques exemples:
$$
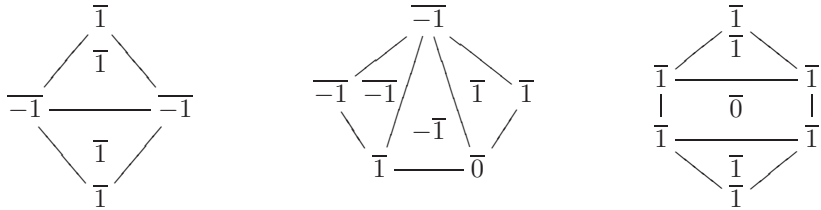
\shorthandoff{; :!?}
\xymatrix @!0 @R=0.60cm @C=1cm
{
&\overline{1}\ar@{-}[ldd]\ar@{-}[rdd]&
\\
&\overline{1}&
\\
\overline{-1}\ar@{-}[rr]&& \overline{-1}
\\
&\overline{1}&
\\
&\overline{1}\ar@{-}[luu]\ar@{-}[ruu]&
}
\qquad
\qquad
\xymatrix @!0 @R=0.50cm @C=0.65cm
{
&&\overline{-1}\ar@{-}[rrdd]\ar@{-}[lldd]\ar@{-}[ldddd]\ar@{-}[rdddd]&
\\
\\
\overline{-1}\ar@{-}[rdd]&\overline{-1}&&\overline{1}& \overline{1}\ar@{-}[ldd]
\\
&&-\overline{1}
\\
&\overline{1}\ar@{-}[rr]&& \overline{0}
}
\qquad
\qquad
\xymatrix @!0 @R=0.40cm @C=0.5cm
{
&&\overline{1}\ar@{-}[rrdd]\ar@{-}[lldd]&
\\
&&\overline{1}&
\\
\overline{1}\ar@{-}[dd]\ar@{-}[rrrr]&&&& \overline{1}\ar@{-}[dd]
\\
&&\overline{0}
\\
\overline{1}\ar@{-}[rrrr]&&&& \overline{1}
\\
&&\overline{1}&
\\
&&\overline{1}\ar@{-}[lluu]\ar@{-}[rruu]
}
$$
}
\end{ex}

Pour relier les solutions de $(E_{3})$ aux découpages de polygones on a besoin d'interpréter géométriquement la somme de deux solutions (voir aussi \cite{C} section 4). Soit $(\overline{a_{1}},\ldots,\overline{a_{n}})$ la quiddité d'une décomposition pondérée de type (3|4) de première espèce d'un polygone convexe $P$ à $n$ sommets.

\begin{itemize}
\item Si $\epsilon \in \{\pm 1\}$ alors $(\overline{a_{1}},\ldots,\overline{a_{n}}) \oplus (\overline{\epsilon},\overline{\epsilon},\overline{\epsilon})$ est la quiddité de la décomposition pondérée de type (3|4) de première espèce du polygone convexe à $(n+1)$ sommets obtenue en rajoutant un triangle de poids $\overline{\epsilon}$ sur le segment reliant le sommet 1 de $P$ au sommet $n$ de $P$.
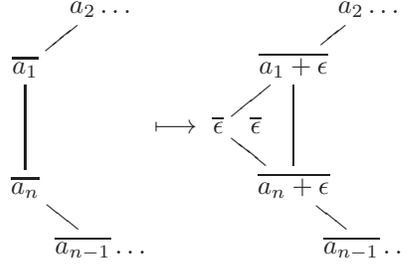
$$
\shorthandoff{; :!?}
\qquad
\xymatrix @!0 @R=0.40cm @C=0.50cm
{
&&\overline{a_{2}}\ldots\ar@{-}[lldd]&
\\
&&& 
\\
\overline{a_{1}}\ar@{-}[dddd]&& 
\\
&&&
\\
&&&& \longmapsto
\\
&&& 
\\
\overline{a_{n}} 
\\
&& 
\\
&&\overline{a_{n-1}}\ldots\ar@{-}[lluu]
}  \xymatrix @!0 @R=0.40cm @C=0.50cm
{
&&&&\overline{a_{2}}\ldots\ar@{-}[lldd]&
\\
&&&&&
\\
&&\overline{a_{1}+\epsilon}\ar@{-}[dddd]&& 
\\
&&&&&
\\
\overline{\epsilon}\ar@{-}[rruu]\ar@{-}[rrdd]&\overline{\epsilon}  
\\
&&&&&
\\
&&\overline{a_{n}+\epsilon} 
\\
&&&&
\\
&&&&\overline{a_{n-1}}\ldots\ar@{-}[lluu]
}
$$
\item $(\overline{a_{1}},\ldots,\overline{a_{n}}) \oplus (\overline{0},\overline{0},\overline{0},\overline{0})$ est la quiddité de la décomposition pondérée de type (3|4) de première espèce du polygone convexe à $(n+2)$ sommets obtenue en rajoutant un quadrilatère de poids $\overline{0}$ sur le segment reliant le sommet 1 de $P$ au sommet $n$ de $P$.
$$
\shorthandoff{; :!?}
\qquad
\xymatrix @!0 @R=0.40cm @C=0.50cm
{
&&
\\
&&\overline{a_{2}}\ldots\ar@{-}[lldd]&
\\
&&& 
\\
\overline{a_{1}}\ar@{-}[dddd]&& 
\\
&&&
\\
&&&& \longmapsto
\\
&&& 
\\
\overline{a_{n}} 
\\
&& 
\\
&&\overline{a_{n-1}}\ldots\ar@{-}[lluu]
}  
\xymatrix @!0 @R=0.40cm @C=0.50cm
{
&&&&\overline{a_{2}}\ldots\ar@{-}[lldd]&
\\
&&&&&
\\
&&\overline{a_{1}}\ar@{-}[dddddd]&& 
\\
&&&&&
\\
\overline{0}\ar@{-}[dd]\ar@{-}[rruu]&
\\
& \overline{0}
\\
\overline{0}\ar@{-}[rrdd]&
\\
&&&&&
\\
&&\overline{a_{n}} 
\\
&&&&
\\
&&&&\overline{a_{n-1}}\ldots\ar@{-}[lluu]
}
$$
\end{itemize}

\begin{thm}

Soit $n \geq 3$. 
\\ i) Toute solution de $(E_{3})$ de taille $n$ est la quiddité associée à une décomposition pondérée de type (3|4) de première espèce d'un polygone convexe à $n$ sommets.
\\ ii) Toute quiddité associée à une décomposition pondérée de type (3|4) de première espèce d'un polygone convexe à $n$ sommets est une solution de taille $n$ de $(E_{3})$.

\end{thm}

\begin{proof}

i) On raisonne par récurrence sur $n$.
\\
\\Si $n=3$ on a deux solutions $(\overline{1},\overline{1},\overline{1})$ qui est la quiddité associée à un triangle de poids $\overline{1}$ et $(\overline{-1},\overline{-1},\overline{-1})$ qui est la quiddité associée à un triangle de poids $\overline{-1}$.
\\
\\Si $n=4$ on a (à permutations cycliques près) trois solutions:
\begin{itemize}
\item $(\overline{0},\overline{0},\overline{0},\overline{0})$ qui est la quiddité associée à un quadrilatère de poids $\overline{0}$. 
\item $(\overline{1},\overline{-1},\overline{1},\overline{-1})$ qui est la quiddité associée à un quadrilatère  découpé en deux triangles de poids $\overline{1}$. 
\item $(\overline{0},\overline{-1},\overline{0},\overline{1})$ qui est la quiddité associée à un quadrilatère  découpé en un triangle de poids $\overline{1}$ et un triangle de poids $\overline{-1}$. 
\\
\end{itemize}

Soient $n \geq 5$ et $(\overline{a_{1}},\ldots,\overline{a_{n}})$ une solution de $(E_{3})$. $(\overline{a_{1}},\ldots,\overline{a_{n}})$ est équivalent à la somme d'un $k$-uplet $(\overline{b_{1}},\ldots,\overline{b_{k}})$ ($k=n-1$ ou $k=n-2$) avec une des solutions irréductibles de $(E_{3})$.  $(\overline{b_{1}},\ldots,\overline{b_{k}})$ est toujours solution de $(E_{3})$ (proposition 3.7) donc il correspond par hypothèse de récurrence à une quiddité associée à une décomposition pondérée de type (3|4) de première espèce d'un polygone convexe à $k$ sommets. Par la discussion précédente, $(\overline{a_{1}},\ldots,\overline{a_{n}})$ est aussi associée à une décomposition pondérée de type (3|4) de première espèce d'un polygone convexe à $n$ sommets.
\\
\\ii) On raisonne par récurrence sur $n$.
\\
\\Si $n=3$, les quiddités associées aux décompositions pondérées de type (3|4) de premières espèces sont $(\overline{1},\overline{1},\overline{1})$ et $(\overline{-1},\overline{-1},\overline{-1})$. Ce sont des solutions de $(E_{3})$. Si $n=4$, les quiddités associées aux décompositions pondérées de type (3|4) de premières espèces sont (à permutation cyclique près) $(\overline{0},\overline{0},\overline{0},\overline{0})$, $(\overline{1},\overline{-1},\overline{1},\overline{-1})$ et $(\overline{0},\overline{-1},\overline{0},\overline{1})$ (le découpage d'un carré en deux triangles de poids $\overline{-1}$ donnant aussi $(\overline{1},\overline{-1},\overline{1},\overline{-1})$). Ce sont des solutions de $(E_{3})$.
\\
\\Considérons une décomposition pondérée de type (3|4) de première espèce d'un polygone convexe $P$ à $n$ sommets et $(\overline{a_{1}},\ldots,\overline{a_{n}})$ la quiddité associée.
\\
\\Si $P$ est le seul sous-polygone intervenant dans la décomposition alors $n=4$ ou $n=3$ et donc la quiddité associée à la décomposition est solution de $(E_{3})$.
\\
\\Sinon on peut trouver un sous-polygone dont tous les cotés sauf un sont des cotés de $P$. Ce polygone est soit un quadrilatère (cas 1) soit un triangle de poids $\overline{\epsilon}$ avec $\overline{\epsilon} \in \{\pm \overline{1}\}$ (cas 2). On considère le polygone $P'$ obtenu en ne conservant de ce sous-polygone que le coté qui n'était pas un coté de $P$. La décomposition de $P$ donne alors une décomposition pondérée de type (3|4) de première espèce de $P'$ et la quiddité $(\overline{b_{1}},\ldots,\overline{b_{k}})$ associée à cette décomposition est solution de $(E_{3})$ (par hypothèse de récurrence). Comme $(\overline{a_{1}},\ldots,\overline{a_{n}})$ est équivalente à la somme de $(\overline{b_{1}},\ldots,\overline{b_{k}})$ avec $(\overline{0},\overline{0},\overline{0},\overline{0})$ (dans le cas 1) ou à la somme de $(\overline{b_{1}},\ldots,\overline{b_{k}})$ avec $(\overline{\epsilon},\overline{\epsilon},\overline{\epsilon})$ (dans le cas 2) on a que $(\overline{a_{1}},\ldots,\overline{a_{n}})$ est solution de $(E_{3})$.

\end{proof}

\begin{prop}

Si $(\overline{c_{1}},\ldots,\overline{c_{n}})$ est une solution de $(E_{3})$ et s'il existe un entier $i$ dans $[\![1;n]\!]$ tel que $\overline{c_{i}} \neq \overline{0}$ alors $(\overline{c_{1}},\ldots,\overline{c_{n}})$ est la quiddité d'une décomposition pondérée de type (3|4) de première espèce d'un polygone convexe à $n$ sommets ne contenant que des triangles.

\end{prop}

\begin{proof}

On raisonne par récurrence sur $n$.
\\
\\Si $n=3$ alors $(E_{3})$ a deux solutions $(\overline{1},\overline{1},\overline{1})$ et $(\overline{-1},\overline{-1},\overline{-1})$. $(\overline{1},\overline{1},\overline{1})$ est la quiddité associée un triangle de poids $\overline{1}$ et $(\overline{-1},\overline{-1},\overline{-1})$ est la quiddité associée à un triangle de poids $\overline{-1}$.
\\
\\Supposons qu'il existe un $n \in \mathbb{N}^{*}$, $n \geq 3$, tel que toute solution de $(E_{3})$ de taille $n$ possédant au moins un élément différent de $\overline{0}$ est la quiddité d'une décomposition pondérée de type (3|4) de première espèce d'un polygone convexe à $n$ sommets ne contenant que des triangles.
\\
\\Soit $(\overline{a_{1}},\ldots,\overline{a_{n+1}})$ une solution de $(E_{3})$ telle qu'il existe un entier $i$ dans $\llbracket 1~;~ n+1 \rrbracket$ tel que $\overline{a_{i}}=\overline{\epsilon}$ avec $\epsilon \in \{\pm 1\}$. On a \[(\overline{a_{1}},\ldots,\overline{a_{n+1}}) \sim (\overline{a_{i+1}-\epsilon},\ldots,\overline{a_{n+1}},\overline{a_{1}},\ldots,\overline{a_{i-1}-\epsilon}) \oplus (\overline{\epsilon},\overline{\epsilon},\overline{\epsilon}).\] Donc, par la proposition 3.7, $(\overline{a_{i+1}-\epsilon},\ldots,\overline{a_{n}},\overline{a_{1}},\ldots,\overline{a_{i-1}-\epsilon})$ est une solution de $(E_{3})$ et donc par invariance circulaire $(\overline{a_{1}},\ldots,\overline{a_{i-1}-\epsilon},\overline{a_{i+1}-\epsilon},\ldots,\overline{a_{n+1}})$ est une solution de $(E_{3})$. On a deux cas:
\\
\\A) $(\overline{a_{1}},\ldots,\overline{a_{i-1}-\epsilon},\overline{a_{i+1}-\epsilon},\ldots,\overline{a_{n+1}})$ possède au moins un élément différent de $\overline{0}$. Par hypothèse de récurrence, ce $n$-uplet est la quiddité d'une décomposition pondérée de type (3|4) de première espèce d'un polygone convexe à $n$ sommets $P$ ne contenant que des triangles. $(\overline{a_{1}},\ldots,\overline{a_{n+1}})$ est la quiddité d'une décomposition pondérée de type (3|4) de première espèce d'un polygone convexe à $(n+1)$ sommets ne contenant que des triangles construit en rajoutant un triangle de poids $\overline{\epsilon}$ sur le segment reliant le sommet $i-1$ de $P$ au sommet $i$ de $P$.
\\
\\B) $(\overline{a_{1}},\ldots,\overline{a_{i-1}-\epsilon},\overline{a_{i+1}-\epsilon},\ldots,\overline{a_{n+1}})$ ne contient que $\overline{0}$. Dans ce cas, $(\overline{a_{1}},\ldots,\overline{a_{n+1}})=(\overline{0},\ldots,\overline{0},\overline{\epsilon},\overline{\epsilon},\overline{\epsilon},\overline{0},\ldots,\overline{0})$ c'est-à-dire $\overline{a_{i}}=\overline{a_{i-1}}=\overline{a_{i+1}}=\overline{\epsilon}$ et tous les autres $\overline{a_{j}}$ sont égaux à $\overline{0}$. On a \[(\overline{a_{1}},\ldots,\overline{a_{n+1}}) \sim (\overline{a_{i+2}-\epsilon},\ldots,\overline{a_{n+1}},\overline{a_{1}},\ldots,\overline{a_{i-1}},\overline{a_{i}-\epsilon}) \oplus (\overline{\epsilon},\overline{\epsilon},\overline{\epsilon}).\] Ainsi, $(\overline{a_{1}},\ldots,\overline{a_{i-1}},\overline{a_{i}-\epsilon},\overline{a_{i+2}-\epsilon},\ldots,\overline{a_{n+1}})$ est une solution de $(E_{3})$ contenant $\overline{a_{i-1}}=\overline{\epsilon}$. Donc, on peut procéder comme en A).

\end{proof}

\begin{exe}
{\rm Par exemple,
$$
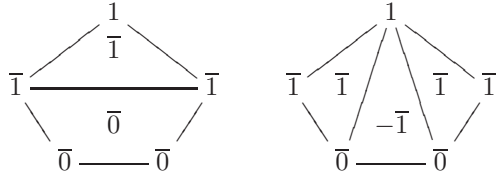
\shorthandoff{; :!?}
\qquad
\xymatrix @!0 @R=0.50cm @C=0.65cm
{
&&\overline{1}\ar@{-}[rrdd]\ar@{-}[lldd]&
\\
&&\overline{1}
\\
\overline{1}\ar@{-}[rdd]\ar@{-}[rrrr]&&&& \overline{1}\ar@{-}[ldd]\\
&&\overline{0}&
\\
&\overline{0}\ar@{-}[rr]&& \overline{0}
}
\qquad
\xymatrix @!0 @R=0.50cm @C=0.65cm
{
&&\overline{1}\ar@{-}[rrdd]\ar@{-}[lldd]\ar@{-}[ldddd]\ar@{-}[rdddd]&
\\
\\
\overline{1}\ar@{-}[rdd]&\overline{1}&&\overline{1}& \overline{1}\ar@{-}[ldd]
\\
&&-\overline{1}
\\
&\overline{0}\ar@{-}[rr]&& \overline{0}
}
$$
}

\end{exe}

\subsection{N=4}
 
\leavevmode\par
\leavevmode\par On étudie maintenant le cas $N=4$.

\subsubsection{Démonstration du théorème 2.5 iii)}

\begin{proof}

Par les propositions 3.4 et 3.8, $(\overline{1},\overline{1},\overline{1})$, $(\overline{-1},\overline{-1},\overline{-1})$, $(\overline{0},\overline{0},\overline{0},\overline{0})$, $(\overline{0},\overline{2},\overline{0},\overline{2})$ ,$(\overline{2},\overline{0},\overline{2},\overline{0})$ et $(\overline{2},\overline{2},\overline{2},\overline{2})$ sont irréductibles. Par les propositions 3.2, 3.3, 3.4 et 3.8, il n'y a pas d'autres solutions irréductibles pour $n=3, 4$. Soient $n \geq 5$ et $(\overline{a_{1}},\ldots,\overline{a_{n}})$ une solution de $(E_{4})$. On a deux cas:
\begin{itemize}
\item Si $(\overline{a_{1}},\ldots,\overline{a_{n}})$ contient $\overline{0}$, $\overline{1}$ ou $\overline{-1}$ alors, par la proposition 3.8, $(\overline{a_{1}},\ldots,\overline{a_{n}})$ est réductible.
\item Si $(\overline{a_{1}},\ldots,\overline{a_{n}})$ ne contient pas  $\overline{0}$, $\overline{1}$ ou $\overline{-1}$ alors $\forall i \in [\![1;n]\!]$ $\overline{a_{i}}=\overline{2}$ et donc $(\overline{a_{1}},\ldots,\overline{a_{n}})$ est réductible puisqu'on peut l'écrire comme la somme du $(n-2)$-uplet $(\overline{0},\overline{2},\ldots,\overline{2},\overline{0})$ ($n-2 \geq 3$) avec $(\overline{2},\overline{2},\overline{2},\overline{2})$.
\end{itemize}

\end{proof}

\subsubsection{Description combinatoire des solutions}

\begin{defn}

i) On appelle décomposition pondérée de type (3|4) de seconde espèce le découpage d'un polygone convexe à $n$ sommets par des diagonales ne se coupant qu'aux sommets et tel que les sous-polygones soient des triangles de poids $\overline{1}$ ou $\overline{-1}$, des quadrilatères de poids $\overline{0}$ ou $\overline{2}$ ou des quadrilatères découpés en deux triangles de poids $\overline{2}$. 
\\
\\ii) On choisit un sommet de $P$ que l'on numérote par 1 puis on numérote les autres sommets de $P$ en suivant le sens horaire ou le sens trigonométrique. La quiddité de la décomposition pondérée de type (3|4) de seconde espèce de $P$ est le $n$-uplet $(\overline{c_{1}},\ldots,\overline{c_{n}})$ avec $\overline{c_{i}}$ la somme des poids des sous-polygones utilisant le sommet $i$.

\end{defn}

\begin{rem}
{\rm Si $(\overline{c_{1}},\ldots,\overline{c_{n}})$ est la quiddité de la décomposition pondérée de type (3|4) de seconde espèce espèce de $P$ alors tout $n$-uplet équivalent à $(\overline{c_{1}},\ldots,\overline{c_{n}})$ est aussi la quiddité de cette décomposition de $P$.
}
\end{rem}

\begin{ex}
{\rm Voici quelques exemples:

$$
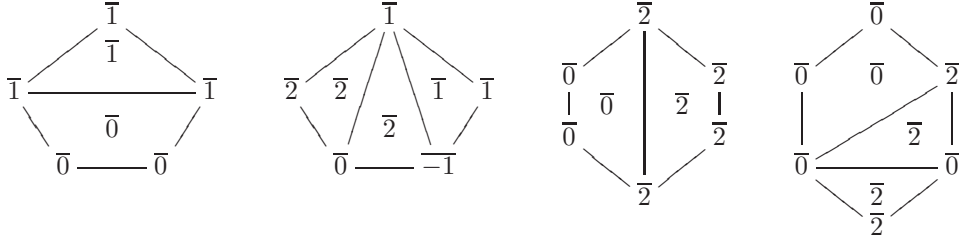
\shorthandoff{; :!?}
\xymatrix @!0 @R=0.50cm @C=0.65cm
{
&&\overline{1}\ar@{-}[rrdd]\ar@{-}[lldd]&
\\
&&\overline{1}
\\
\overline{1}\ar@{-}[rdd]\ar@{-}[rrrr]&&&& \overline{1}\ar@{-}[ldd]\\
&&\overline{0}&
\\
&\overline{0}\ar@{-}[rr]&& \overline{0}
}
\qquad
\xymatrix @!0 @R=0.50cm @C=0.65cm
{
&&\overline{1}\ar@{-}[rrdd]\ar@{-}[lldd]\ar@{-}[ldddd]\ar@{-}[rdddd]&
\\
\\
\overline{2}\ar@{-}[rdd]&\overline{2}&&\overline{1}& \overline{1}\ar@{-}[ldd]
\\
&&\overline{2}
\\
&\overline{0}\ar@{-}[rr]&& \overline{-1}
}
\qquad
\xymatrix @!0 @R=0.40cm @C=0.5cm
{
&&\overline{2}\ar@{-}[rrdd]\ar@{-}[dddddd]\ar@{-}[lldd]&
\\
&&&
\\
\overline{0}\ar@{-}[dd]&&&& \overline{2}\ar@{-}[dd]
\\
& \overline{0} && \overline{2}
\\
\overline{0} &&&& \overline{2}
\\
&&&
\\
&&\overline{2}\ar@{-}[lluu]\ar@{-}[rruu]
}
\qquad
\xymatrix @!0 @R=0.40cm @C=0.5cm
{
&&\overline{0}\ar@{-}[rrdd]\ar@{-}[lldd]&
\\
&&& 
\\
\overline{0}\ar@{-}[ddd]&& \overline{0} && \overline{2}\ar@{-}[ddd]
\\
&
\\
&&& \overline{2}
\\
\overline{0} \ar@{-}[rrrruuu] \ar@{-}[rrrr] &&&& \overline{0}
\\
&& \overline{2}
\\
&&\overline{2}\ar@{-}[lluu]\ar@{-}[rruu]
}
$$
}

\end{ex}

Les considérations géométriques données après la définition 4.5 s'adaptent naturellement au cas des décompositions pondérées de type (3|4) de seconde espèce. Pour relier les solutions de $(E_{4})$ aux découpages de polygones on a besoin en plus des considérations suivantes:

\begin{itemize}
\item $(\overline{a_{1}},\ldots,\overline{a_{n}}) \oplus (\overline{2},\overline{2},\overline{2},\overline{2})$ est la quiddité de la décomposition pondérée de type (3|4) de seconde espèce du polygone convexe à $(n+2)$ sommets obtenue en rajoutant un quadrilatère de poids $\overline{2}$ sur le segment reliant le sommet 1 de $P$ au sommet $n$ de $P$.
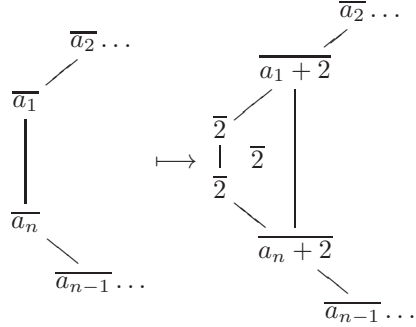
$$
\shorthandoff{; :!?}
\qquad
\xymatrix @!0 @R=0.40cm @C=0.50cm
{
&&
\\
&&\overline{a_{2}}\ldots\ar@{-}[lldd]&
\\
&&& 
\\
\overline{a_{1}}\ar@{-}[dddd]&& 
\\
&&&
\\
&&&& \longmapsto
\\
&&& 
\\
\overline{a_{n}} 
\\
&& 
\\
&&\overline{a_{n-1}}\ldots\ar@{-}[lluu]
}  
\xymatrix @!0 @R=0.40cm @C=0.50cm
{
&&&&\overline{a_{2}}\ldots\ar@{-}[lldd]&
\\
&&&&&
\\
&&\overline{a_{1}+2}\ar@{-}[dddddd]&& 
\\
&&&&&
\\
\overline{2}\ar@{-}[dd]\ar@{-}[rruu]&
\\
& \overline{2}
\\
\overline{2}\ar@{-}[rrdd]&
\\
&&&&&
\\
&&\overline{a_{n}+2} 
\\
&&&&
\\
&&&&\overline{a_{n-1}}\ldots\ar@{-}[lluu]
}
$$ 
\item $(\overline{a_{1}},\ldots,\overline{a_{n}}) \oplus (\overline{0},\overline{2},\overline{0},\overline{2})$ est la quiddité de la décomposition pondérée de type (3|4) de seconde espèce du polygone convexe à $(n+2)$ sommets obtenue en rajoutant un quadrilatère découpés en deux triangles de poids $\overline{2}$ sur le segment reliant le sommet 1 de $P$ au sommet $n$ de $P$.
$$
\shorthandoff{; :!?}
\qquad
\xymatrix @!0 @R=0.40cm @C=0.50cm
{
&&
\\
&&
\\
&&\overline{a_{2}}\ldots\ar@{-}[lldd]&
\\
&&& 
\\
\overline{a_{1}}\ar@{-}[dddd]&& 
\\
&&&
\\
&&&& \longmapsto
\\
&&& 
\\
\overline{a_{n}} 
\\
&& 
\\
&&\overline{a_{n-1}}\ldots\ar@{-}[lluu]
}  
\xymatrix @!0 @R=0.40cm @C=0.50cm
{
&&&&&\overline{a_{2}}\ldots\ar@{-}[lldd]&
\\
&&&&&
\\
&&&\overline{a_{1}+2}\ar@{-}[ddddddd]&& 
\\
&&&&&
\\
\overline{0}\ar@{-}[ddd]\ar@{-}[rrruu] &&\overline{2}
\\
&
\\
& 
\\
\overline{2} \ar@{-}[rrrdd] &\overline{2}
\\
&&&&&
\\
&&&\overline{a_{n}}\ar@{-}[llluuuuu] 
\\
&&&&
\\
&&&&&\overline{a_{n-1}}\ldots\ar@{-}[lluu]
}
$$ 
\end{itemize}

\begin{thm}

Soit $n \geq 3$. 
\\ i) Toute solution de $(E_{4})$ de taille $n$ est la quiddité associée à une décomposition pondérée de type (3|4) de seconde espèce d'un polygone convexe à $n$ sommets.
\\ ii) Toute quiddité associée à une décomposition pondérée de type (3|4) de seconde espèce d'un polygone convexe à $n$ sommets est une solution de taille $n$ de $(E_{4})$.

\end{thm}

\begin{proof}

i) On raisonne par récurrence sur $n$.
\\
\\Si $n=3$, $(E_{4})$ a deux solutions $(\overline{1},\overline{1},\overline{1})$ qui est la quiddité associée à un triangle de poids $\overline{1}$ et $(\overline{-1},\overline{-1},\overline{-1})$ qui est la quiddité associée à un triangle de poids $\overline{-1}$.
\\
\\Si $n=4$, on a (à permutations cycliques près) six solutions $(\overline{0},\overline{0},\overline{0},\overline{0})$, $(\overline{1},\overline{2},\overline{1},\overline{2})$, $(\overline{0},\overline{2},\overline{0},\overline{2})$, $(\overline{0},\overline{1},\overline{0},\overline{-1})$, $(\overline{2},\overline{-1},\overline{2},\overline{-1})$ et $(\overline{2},\overline{2},\overline{2},\overline{2})$ qui sont chacune une quiddité d'une décomposition pondérée de type (3|4) de seconde espèce d'un quadrilatère.
\\
\\Soient $n \geq 5$ et $(\overline{a_{1}},\ldots,\overline{a_{n}})$ une solution de $(E_{4})$. $(\overline{a_{1}},\ldots,\overline{a_{n}})$ est équivalent à la somme d'un $k$-uplet $(\overline{b_{1}},\ldots,\overline{b_{k}})$ ($k=n-1$ ou $k=n-2$) avec une des solutions irréductibles de $(E_{4})$. $(\overline{b_{1}},\ldots,\overline{b_{k}})$ est toujours solution de $(E_{4})$ (proposition 3.7) donc il correspond par hypothèse de récurrence à une quiddité associée à une décomposition pondérée de type (3|4) de seconde espèce d'un polygone convexe à $k$ sommets. Par la discussion géométrique précédente, $(\overline{a_{1}},\ldots,\overline{a_{n}})$ est aussi associée à une décomposition pondérée de type (3|4) de seconde espèce d'un polygone convexe à $n$ sommets.
\\
\\ii) On raisonne par récurrence sur $n$.
\\
\\Si $n=3$, les quiddités associées aux décompositions pondérées de type (3|4) de seconde espèce sont $(\overline{1},\overline{1},\overline{1})$ et $(\overline{-1},\overline{-1},\overline{-1})$. Ce sont des solutions de $(E_{4})$. Si $n=4$, les quiddités associées aux décompositions pondérées de type (3|4) de seconde espèce sont (à permutation cyclique près) $(\overline{0},\overline{0},\overline{0},\overline{0})$, $(\overline{1},\overline{2},\overline{1},\overline{2})$, $(\overline{0},\overline{2},\overline{0},\overline{2})$, $(\overline{0},\overline{1},\overline{0},\overline{-1})$, $(\overline{2},\overline{-1},\overline{2},\overline{-1})$ et $(\overline{2},\overline{2},\overline{2},\overline{2})$. Ce sont des solutions de $(E_{4})$.
\\
\\Considérons une décomposition pondérée de type (3|4) de seconde espèce d'un polygone convexe $P$ à $n$ sommets et $(\overline{a_{1}},\ldots,\overline{a_{n}})$ la quiddité associée.  
\\
\\Si $P$ est le seul sous-polygone intervenant dans la décomposition alors $n=4$ ou $n=3$ et donc la quiddité associée à la décomposition est solution de $(E_{4})$.
\\
\\Sinon on peut trouver un sous-polygone dont tous les cotés sauf un sont des cotés de $P$. Ce polygone est soit un quadrilatère de poids $\overline{0}$ (cas 1) soit un quadrilatère de poids $\overline{2}$ (cas 2) soit un triangle de poids $\overline{\epsilon}$ avec $\epsilon \in \{\pm 1\}$ (cas 3) soit un triangle de poids $\overline{2}$.
\\
\\Si l'on est dans les cas 1, 2 ou 3 . On considère le polygone $P'$ obtenu en ne conservant de ce sous-polygone que le coté qui n'était pas un coté de $P$. La décomposition de $P$ donne alors une décomposition pondérée de type (3|4) de seconde espèce de $P'$ et la quiddité $(\overline{b_{1}},\ldots,\overline{b_{k}})$ associée à cette décomposition est solution de $(E_{4})$ par hypothèse de récurrence. Comme $(\overline{a_{1}},\ldots,\overline{a_{n}})$ est équivalente à la somme de $(\overline{b_{1}},\ldots,\overline{b_{k}})$ avec $(\overline{0},\overline{0},\overline{0},\overline{0})$ (dans le cas 1 ), $(\overline{2},\overline{2},\overline{2},\overline{2})$ (dans le cas 2) ou $(\overline{\epsilon},\overline{\epsilon},\overline{\epsilon})$ (dans le cas 3) on a que $(\overline{a_{1}},\ldots,\overline{a_{n}})$ est solution de $(E_{4})$.
\\
\\Si l'on n'est pas dans les cas 1, 2 ou 3. Il existe un triangle extérieur de poids $\overline{2}$ adjacent à un triangle de poids $\overline{2}$ dont l'un des côtés est un côté de $P$. Si ces deux triangles sont les deux seuls sous-polygones intervenant dans la décomposition de $P$ alors $n=4$ et la quiddité associée est solution de $(E_{4})$. Sinon on considère le polygone $P'$ obtenu en supprimant ces deux triangles. La décomposition de $P$ donne alors une décomposition pondérée de type (3|4) de seconde espèce de $P'$ et la quiddité $(\overline{b_{1}},\ldots,\overline{b_{k}})$ associée à cette décomposition est solution de $(E_{4})$ par hypothèse de récurrence. Comme $(\overline{a_{1}},\ldots,\overline{a_{n}})$ est équivalente à la somme de $(\overline{b_{1}},\ldots,\overline{b_{k}})$ avec $(\overline{0},\overline{2},\overline{0},\overline{2})$ on a que $(\overline{a_{1}},\ldots,\overline{a_{n}})$ est solution de $(E_{4})$.

\end{proof}

On ne peut malheureusement pas étendre la proposition 4.7 pour $N=4$. En effet, $(\overline{2},\overline{2},\overline{2},\overline{2})$ est solution de $(E_{4})$ mais n'est pas la quiddité d'une décomposition pondérée de type (3|4) de seconde espèce d'un quadrilatère ne contenant que des triangles. On dispose cependant du résultat suivant:

\begin{prop}

Si $(\overline{c_{1}},\ldots,\overline{c_{n}})$ est une solution de $(E_{4})$ et s'il existe un entier $i$ dans $[\![1;n]\!]$ tel que $\overline{c_{i}} \in \{\pm \overline{1}\}$ alors $(\overline{c_{1}},\ldots,\overline{c_{n}})$ est la quiddité d'une décomposition pondérée de type (3|4) de seconde espèce d'un polygone convexe à $n$ sommets ne contenant que des triangles.

\end{prop}

\begin{proof}

On raisonne par récurrence sur $n$.
\\
\\Si $n=3$ alors $(E_{4})$ a deux solutions $(\overline{1},\overline{1},\overline{1})$ et $(\overline{-1},\overline{-1},\overline{-1})$. $(\overline{1},\overline{1},\overline{1})$ est la quiddité associée à un triangle de poids $\overline{1}$ et $(\overline{-1},\overline{-1},\overline{-1})$ est la quiddité associée à un triangle de poids $\overline{-1}$.
\\
\\Supposons qu'il existe un $n \in \mathbb{N}^{*}$ $n \geq 3$ tel que toute solution de $(E_{4})$ de taille $n$ possédant au moins un élément valant $\pm \overline{1}$ est est la quiddité d'une décomposition pondérée de type (3|4) de seconde espèce d'un polygone convexe à $n$ sommets ne contenant que des triangles.
\\
\\Soit $(\overline{a_{1}},\ldots,\overline{a_{n+1}})$ une solution de $(E_{4})$ possédant au moins un élément valant $\pm \overline{1}$. $\exists i \in [\![1;n+1]\!]$ tel que $\overline{a_{i}}=\overline{\epsilon}$ avec $\epsilon \in \{\pm 1\}$. On a \[(\overline{a_{1}},\ldots,\overline{a_{n+1}}) \sim (\overline{a_{i+1}-\epsilon},\ldots,\overline{a_{n+1}},\overline{a_{1}},\ldots,\overline{a_{i-1}-\epsilon}) \oplus (\overline{\epsilon},\overline{\epsilon},\overline{\epsilon}).\] Donc, par la proposition 3.7, $(\overline{a_{i+1}-\epsilon},\ldots,\overline{a_{n+1}},\overline{a_{1}},\ldots,\overline{a_{i-1}-\epsilon})$ est une solution de $(E_{4})$ et donc par invariance circulaire $(\overline{a_{1}},\ldots,\overline{a_{i-1}-\epsilon},\overline{a_{i+1}-\epsilon},\ldots,\overline{a_{n+1}})$ est une solution de $(E_{4})$. On a deux cas:
\\
\\A) $(\overline{a_{1}},\ldots,\overline{a_{i-1}-\epsilon},\overline{a_{i+1}-\epsilon},\ldots,\overline{a_{n+1}})$ possède au moins un élément valant $\pm \overline{1}$. Par hypothèse de récurrence, ce $n$-uplet est la quiddité d'une décomposition pondérée de type (3|4) de seconde espèce d'un polygone convexe à $n$ sommets $P$ ne contenant que des triangles. $(\overline{a_{1}},\ldots,\overline{a_{n+1}})$ est la quiddité d'une décomposition pondérée de type (3|4) de seconde espèce d'un polygone convexe à $(n+1)$ sommets ne contenant que des triangles construit en rajoutant un triangle de poids $\overline{\epsilon}$ sur le segment reliant le sommet $i-1$ de $P$ au sommet $i$ de $P$.
\\
\\B) $(\overline{a_{1}},\ldots,\overline{a_{i-1}-\epsilon},\overline{a_{i+1}-\epsilon},\ldots,\overline{a_{n+1}})$ ne contient pas d'élément valant $\pm \overline{1}$. Dans ce cas, $\overline{a_{i-1}}=\overline{\alpha} \in \{\pm \overline{1}\}$ et $\overline{a_{i+1}}=\overline{\beta}\in \{\pm \overline{1}\}$ (car $\overline{a_{i-1}-\epsilon} \in \{\overline{0},\overline{2}\}$ et $\overline{a_{i+1}-\epsilon} \in \{\overline{0},\overline{2}\}$) et tous les autres $\overline{a_{j}}$ ($j \neq i, i-1, i+1$) sont égaux à $\overline{0}$ ou $\overline{2}$.  On a \[(\overline{a_{1}},\ldots,\overline{a_{n+1}}) \sim (\overline{a_{i+2}-\beta},\ldots,\overline{a_{n+1}},\overline{a_{1}},\ldots,\overline{a_{i-1}},\overline{a_{i}-\beta}) \oplus (\overline{\beta},\overline{\beta},\overline{\beta}).\] Ainsi, $(\overline{a_{1}},\ldots,\overline{a_{i-1}},\overline{a_{i}-\beta},\overline{a_{i+2}-\beta},\ldots,\overline{a_{n+1}})$ est une solution de $(E_{4})$ contenant $\overline{a_{i-1}}=\overline{\alpha}$. Donc, on peut procéder comme en A).

\end{proof}

\subsection{Cas $N=5$}

\begin{proof}[Démonstration du théorème 2.5 iv)]

Par les propositions 3.2, 3.3, 3.4 et 3.8, les seules solutions irréductibles de $(E_{5})$ de taille 3 et 4 sont celles données dans l'énoncé. $(\overline{2},\overline{2},\overline{2},\overline{2},\overline{2})$ est une solution irréductible (Théorème 2.6) et $(\overline{3},\overline{3},\overline{3},\overline{3},\overline{3})$ est une solution irréductible (Corollaire 3.22). 
\\Un simple calcul montre que $(\overline{3},\overline{2},\overline{2},\overline{3},\overline{2},\overline{2})$, $(\overline{2},\overline{3},\overline{3},\overline{2},\overline{3},\overline{3})$ et $(\overline{2},\overline{3},\overline{2},\overline{3},\overline{2},\overline{3})$ sont solutions de $(E_{5})$. Supposons par l'absurde $(\overline{3},\overline{2},\overline{2},\overline{3},\overline{2},\overline{2})$ réductible. Comme $(\overline{3},\overline{2},\overline{2},\overline{3},\overline{2},\overline{2})$ ne contient pas $\overline{1}$ ou $\overline{-1}$, elle est équivalente à la somme de deux solutions de taille 4 ne contenant pas $\overline{1}$ ou $\overline{-1}$. Elle contient alors nécessairement $\overline{0}$ ce qui est absurde. On montre de la même façon que $(\overline{2},\overline{3},\overline{3},\overline{2},\overline{3},\overline{3})$ et $(\overline{2},\overline{3},\overline{2},\overline{3},\overline{2},\overline{3})$ sont irréductibles.
\\
\\Soit $(\overline{a_{1}},\ldots,\overline{a_{n}})$ une solution de $(E_{5})$.
\\
\\ \underline{Si $n=5$ :} S'il existe un entier $i$ dans $[\![1;n]\!]$ tel que $\overline{a_{i}} \in \{\overline{1}, \overline{-1}, \overline{0}\}$ alors $(\overline{a_{1}},\ldots,\overline{a_{5}})$ est réductible par la proposition 3.8. Si $\forall i \in [\![1;n]\!]$ $\overline{a_{i}} \notin \{\overline{1}, \overline{-1}, \overline{0}\}$ alors les $\overline{a_{i}}$ valent $\overline{2}$ ou $\overline{3}$. On a donc 32 possibilités et si on effectue le calcul pour chacune de ces possibilités on trouve seulement 2 solutions: $(\overline{2},\overline{2},\overline{2},\overline{2},\overline{2})$ et $(\overline{3},\overline{3},\overline{3},\overline{3},\overline{3})$.
\\
\\ \underline{Si $n \geq 6$ :} Si un des $\overline{a_{i}}$ est égal à $\overline{0}$, $\overline{1}$ ou $\overline{-1}$ alors $(\overline{a_{1}},\ldots,\overline{a_{n}})$ est réductible par la proposition 3.8.
\\
\\ Si $\forall i \in [\![1;n]\!]$ $\overline{a_{i}} \notin \{\overline{1}, \overline{-1}, \overline{0}\}$. On a plusieurs cas:
\begin{itemize}
\item Si $(\overline{a_{1}},\ldots,\overline{a_{n}})$ contient trois $\overline{2}$ (respectivement $\overline{3}$) consécutifs alors $(\overline{a_{1}},\ldots,\overline{a_{n}})$ est réductible puisqu'il est équivalent à la somme d'un $(n-3)$-uplet ($n-3 \geq 3$) avec $(\overline{2},\overline{2},\overline{2},\overline{2},\overline{2})$ (respectivement $(\overline{3},\overline{3},\overline{3},\overline{3},\overline{3})$).
\\
\item Si $(\overline{a_{1}},\ldots,\overline{a_{n}})$ ne contient ni trois $\overline{2}$ consécutifs ni trois $\overline{3}$ consécutifs mais contient deux $\overline{2}$ (respectivement $\overline{3}$) consécutifs. Dans ce cas, $(\overline{a_{1}},\ldots,\overline{a_{n}})$ contient $(\overline{3},\overline{2},\overline{2},\overline{3})$ (respectivement $(\overline{2},\overline{3},\overline{3},\overline{2})$). Donc, $(\overline{a_{1}},\ldots,\overline{a_{n}})$ est équivalent à la somme d'un $(n-4)$-uplet avec $(\overline{2},\overline{3},\overline{2},\overline{2},\overline{3},\overline{2})$ (respectivement $(\overline{3},\overline{2},\overline{3},\overline{3},\overline{2},\overline{3}$). Si $n=6$ alors $n-4=2$ et donc $(\overline{a_{1}},\ldots,\overline{a_{n}})$ est équivalent à $(\overline{2},\overline{3},\overline{2},\overline{2},\overline{3},\overline{2})$ (respectivement $(\overline{3},\overline{2},\overline{3},\overline{3},\overline{2},\overline{3}$)). Si $n \geq 7$ alors $n-4 \geq 3$ et $(\overline{a_{1}},\ldots,\overline{a_{n}})$ est réductible.
\\
\item Si on n'est dans aucun de ces deux cas alors $n$ est pair et $(\overline{a_{1}},\ldots,\overline{a_{n}})$ est équivalent au $n$-uplet constitué de la répétition de $(\overline{2},\overline{3})$. Si $n=6$ alors $(\overline{a_{1}},\ldots,\overline{a_{n}})$ est équivalent à $(\overline{2},\overline{3},\overline{2},\overline{3},\overline{2},\overline{3})$. Si $n \geq 7$ alors $(\overline{a_{1}},\ldots,\overline{a_{n}})$ est équivalent à la somme d'un $(n-4)$-uplet avec $(\overline{3},\overline{2},\overline{3},\overline{2},\overline{3},\overline{2})$ et $n-4 \geq 3$. Donc, $(\overline{a_{1}},\ldots,\overline{a_{n}})$ est réductible.

\end{itemize}

\end{proof}

\subsection{cas $N=6$}

\begin{proof}[Démonstration du Théorème 2.5 v)]

Par les propositions 3.2, 3.3, 3.4 et 3.8, les seules solutions irréductibles de $(E_{6})$ pour $n=3$ et $n=4$ sont celles données dans l'énoncé. $(\overline{2},\overline{2},\overline{2},\overline{2},\overline{2},\overline{2})$ est une solution irréductible (Théorème 2.6) et $(\overline{4},\overline{4},\overline{4},\overline{4},\overline{4},\overline{4})$ est une solution irréductible (Corollaire 3.22). On vérifie que $(\overline{3},\overline{3},\overline{3},\overline{3},\overline{3},\overline{3})$ est solution. De plus, celle-ci est irréductible (car $(E_{6})$ n'a pas de solutions de la forme $(\overline{a},\overline{3},\overline{b})$ ou $(\overline{a},\overline{3},\overline{3},\overline{b})$).
\\
\\Soit $n \geq 5$. Soit $(\overline{a_{1}},\ldots,\overline{a_{n}})$ une solution de $(E_{6})$.
\\
\\A) Si un des $\overline{a_{i}}$ est égal à $\overline{0}$, $\overline{1}$ ou $\overline{-1}$ alors $(\overline{a_{1}},\ldots,\overline{a_{n}})$ est réductible par la proposition 3.8.
\\
\\B) Sinon les $\overline{a_{i}}$ ne peuvent valoir que $\overline{2}$, $\overline{3}$ ou $\overline{4}$. On a trois cas:
\\
\\i)Il existe un entier $i$ dans $[\![1;n]\!]$ tel que $\overline{a_{i}}=\overline{2}$.
\\
\\S'il existe un entier $i$ dans $[\![1;n]\!]$ tel que $\overline{a_{i}}=\overline{2}$ et $\overline{a_{i-1}} \neq \overline{2}$ ou $\overline{a_{i+1}} \neq \overline{2}$ alors $(\overline{a_{1}},\ldots,\overline{a_{n}})$ est réductible. En effet, 
si $\overline{a_{i+1}}=\overline{3}$ alors \[(\overline{a_{i+2}},\ldots,\overline{a_{n}},\overline{a_{1}},\ldots,\overline{a_{i-1}},\overline{a_{i}},\overline{a_{i+1}})= (\overline{a_{i+2}-4},\ldots,\overline{a_{n}},\overline{a_{1}},\ldots,\overline{a_{i-1}-3}) \oplus (\overline{3},\overline{2},\overline{3},\overline{4}).\] On procède de façon analogue si $\overline{a_{i-1}}=\overline{3}$. Si $\overline{a_{i+1}}=\overline{4}$ alors \[(\overline{a_{i+2}},\ldots,\overline{a_{n}},\overline{a_{1}},\ldots,\overline{a_{i-1}},\overline{a_{i}},\overline{a_{i+1}})= (\overline{a_{i+2}-2},\ldots,\overline{a_{n}},\overline{a_{1}},\ldots,\overline{a_{i-1}-4}) \oplus (\overline{4},\overline{2},\overline{4},\overline{2}).\] On procède de façon analogue si $\overline{a_{i-1}}=\overline{4}$.
\\
\\Sinon tous les $\overline{a_{i}}$ sont égaux à $\overline{2}$. Dans ce cas, on a $n \equiv 0~[N]$ par le lemme 3.21. Si $n=6$ alors $(\overline{a_{1}},\ldots,\overline{a_{n}})$ est irréductible (Théorème 2.6) et sinon $n \geq 12$ et $(\overline{a_{1}},\ldots,\overline{a_{n}})$ est réductible puisqu'on peut l'écrire comme la somme du $(n-4)$-uplet $(\overline{0},\overline{2},\ldots,\overline{2},\overline{0})$ avec $(\overline{2},\overline{2},\overline{2},\overline{2},\overline{2},\overline{2})$ et $n-4 \geq 3$. 
\\
\\ ii) Pour tout entier $i$ compris entre 1 et $n$, $\overline{a_{i}} \in \{\overline{3},\overline{4}\}$ et il existe un entier $i$ dans $[\![1;n]\!]$ tel que $\overline{a_{i}}=\overline{3}$.
\\
\\S'il existe un entier $i$ dans $[\![1;n]\!]$ tel que $\overline{a_{i}}=\overline{3}$ et $\overline{a_{i-1}}=\overline{4}$ ou $\overline{a_{i+1}}=\overline{4}$. Dans ce cas, la solution est réductible. En effet, si $\overline{a_{i+1}}=\overline{4}$ alors \[(\overline{a_{i+2}},\ldots,\overline{a_{n}},\overline{a_{1}},\ldots,\overline{a_{i-1}},\overline{a_{i}},\overline{a_{i+1}})= (\overline{a_{i+2}-3},\ldots,\overline{a_{n}},\overline{a_{1}},\ldots,\overline{a_{i-1}-2}) \oplus (\overline{2},\overline{3},\overline{4},\overline{3}).\] On procède de façon analogue si $\overline{a_{i-1}}=\overline{4}$.
\\
\\Sinon tous les $\overline{a_{i}}$ sont égaux à $\overline{3}$. Comme $M_{5}(\overline{3},\overline{3},\overline{3},\overline{3},\overline{3})=\begin{pmatrix}
    \overline{0} & \overline{-1} \\
    \overline{1}    & \overline{3} 
    \end{pmatrix} \neq \pm Id$, on a $n \geq 6$. Si $n=6$ alors la solution est irréductible. Si $n \geq 7$ alors la solution est réductible puisqu'on peut l'écrire comme la somme du $(n-4)$-uplet $(\overline{0},\overline{3},\ldots,\overline{3},\overline{0})$ avec $(\overline{3},\overline{3},\overline{3},\overline{3},\overline{3},\overline{3})$ et $n-4 \geq 3$.
\\
\\iii) Pour tout entier $i$ compris entre 1 et $n$, $\overline{a_{i}}=\overline{4}$. On a $n \geq 6$ car sinon $(\overline{-a_{1}},\ldots,\overline{-a_{n}})$ est un 5-uplet solution ne contenant que des $\overline{2}$ ce qui est impossible (lemme 3.21). Si $n=6$ alors la solution est irréductible (corollaire 3.22). Si $n \geq 7$ alors la solution est réductible puisqu'on peut l'écrire comme la somme du $(n-4)$-uplet $(\overline{0},\overline{4},\ldots,\overline{4},\overline{0})$ avec $(\overline{4},\overline{4},\overline{4},\overline{4},\overline{4},\overline{4})$ et $n-4 \geq 3$.
\end{proof}

\begin{rems}

{\rm Le cas $N=6$ montre en particulier qu'il peut exister un entier $n$ tel que $(E_{N})$ ne possède pas de solution irréductible de taille $n$ mais en possède de taille strictement supérieure à $n$.
}

\end{rems}

\subsection{Cas $N=7$} 

\begin{thm} Les solutions irréductibles de $(E_{7})$ sont (à permutations cycliques prés) :
\begin{itemize}

\item $n=3$: $(\overline{1},\overline{1},\overline{1})$ {\rm et} $(\overline{-1},\overline{-1},\overline{-1})$
\\

\item $n=4$: \\ $(\overline{3},\overline{3},\overline{3},\overline{3})$, $(\overline{4},\overline{4},\overline{4},\overline{4})$,
\\$(\overline{5},\overline{0},\overline{2},\overline{0})$, 
\\$(\overline{4},\overline{0},\overline{3},\overline{0})$,
\\$(\overline{0},\overline{0},\overline{0},\overline{0})$,
\\

\item $n=5$: $(\overline{2},\overline{2},\overline{5},\overline{4},\overline{5})$, $(\overline{5},\overline{5},\overline{2},\overline{3},\overline{2})$
\\

\item $n=6$
\\ $(\overline{2},\overline{2},\overline{2},\overline{4},\overline{3},\overline{4})$, $(\overline{5},\overline{5},\overline{5},\overline{3},\overline{4},\overline{3})$,
\\$(\overline{2},\overline{3},\overline{4},\overline{5},\overline{4},\overline{3})$, 
\\$(\overline{2},\overline{3},\overline{5},\overline{2},\overline{5},\overline{3})$, $(\overline{5},\overline{4},\overline{2},\overline{5},\overline{2},\overline{4})$, 
\\$(\overline{2},\overline{3},\overline{5},\overline{3},\overline{2},\overline{4})$, $(\overline{5},\overline{4},\overline{2},\overline{4},\overline{5},\overline{3})$, 
\\$(\overline{2},\overline{4},\overline{2},\overline{4},\overline{2},\overline{4})$, $(\overline{5},\overline{3},\overline{5},\overline{3},\overline{5},\overline{3})$, 
\\$(\overline{2},\overline{5},\overline{2},\overline{5},\overline{2},\overline{5})$
\\                                                     

\item $n=7$:
\\ $(\overline{2},\overline{2},\overline{2},\overline{2},\overline{2},\overline{2},\overline{2})$,    $(\overline{5},\overline{5},\overline{5},\overline{5},\overline{5},\overline{5},\overline{5})$,
\\$(\overline{2},\overline{2},\overline{2},\overline{3},\overline{5},\overline{5},\overline{3})$,  $(\overline{5},\overline{5},\overline{5},\overline{4},\overline{2},\overline{2},\overline{4})$,
\\$(\overline{2},\overline{2},\overline{3},\overline{4},\overline{2},\overline{4},\overline{3})$, $(\overline{5},\overline{5},\overline{4},\overline{3},\overline{5},\overline{3},\overline{4})$
\\

\item $n=8$:
\\ $(\overline{2},\overline{2},\overline{3},\overline{4},\overline{3},\overline{2},\overline{2},\overline{4})$, $(\overline{5},\overline{5},\overline{4},\overline{3},\overline{4},\overline{5},\overline{5},\overline{3})$,
\\$(\overline{2},\overline{3},\overline{4},\overline{3},\overline{4},\overline{5},\overline{3},\overline{4})$, $(\overline{5},\overline{4},\overline{3},\overline{4},\overline{3},\overline{2},\overline{4},\overline{3})$,
\\$(\overline{2},\overline{4},\overline{3},\overline{5},\overline{2},\overline{4},\overline{3},\overline{5})$, $(\overline{5},\overline{3},\overline{4},\overline{2},\overline{5},\overline{3},\overline{4},\overline{2})$,
\\
$(\overline{3},\overline{4},\overline{3},\overline{4},\overline{3},\overline{4},\overline{3},\overline{4})$
\\

\item $n=9$: 
\\$(\overline{2},\overline{2},\overline{2},\overline{2},\overline{3},\overline{4},\overline{3},\overline{4},\overline{3})$, $(\overline{5},\overline{5},\overline{5},\overline{5},\overline{4},\overline{3},\overline{4},\overline{3},\overline{4})$,
\\$(\overline{2},\overline{2},\overline{3},\overline{5},\overline{4},\overline{3},\overline{4},\overline{5},\overline{3})$, $(\overline{5},\overline{5},\overline{4},\overline{2},\overline{3},\overline{4},\overline{3},\overline{2},\overline{4})$,
\\$(\overline{2},\overline{2},\overline{3},\overline{5},\overline{4},\overline{3},\overline{5},\overline{2},\overline{4})$, $(\overline{5},\overline{5},\overline{4},\overline{2},\overline{3},\overline{4},\overline{2},\overline{5},\overline{3})$,
\\$(\overline{2},\overline{2},\overline{4},\overline{2},\overline{2},\overline{4},\overline{2},\overline{2},\overline{4})$, $(\overline{5},\overline{5},\overline{3},\overline{5},\overline{5},\overline{3},\overline{5},\overline{5},\overline{3})$,
\\$(\overline{2},\overline{2},\overline{4},\overline{2},\overline{5},\overline{3},\overline{4},\overline{5},\overline{3})$, $(\overline{5},\overline{5},\overline{3},\overline{5},\overline{2},\overline{4},\overline{3},\overline{2},\overline{4})$,
\\$(\overline{2},\overline{2},\overline{4},\overline{2},\overline{5},\overline{3},\overline{5},\overline{2},\overline{4})$, $(\overline{5},\overline{5},\overline{3},\overline{5},\overline{2},\overline{4},\overline{2},\overline{5},\overline{3})$,
\\$(\overline{2},\overline{3},\overline{4},\overline{2},\overline{3},\overline{4},\overline{2},\overline{3},\overline{4})$, $(\overline{5},\overline{4},\overline{3},\overline{5},\overline{4},\overline{3},\overline{5},\overline{4},\overline{3})$,
\\$(\overline{2},\overline{4},\overline{3},\overline{2},\overline{4},\overline{3},\overline{2},\overline{4},\overline{3})$, $(\overline{5},\overline{3},\overline{4},\overline{5},\overline{3},\overline{4},\overline{5},\overline{3},\overline{4})$.
\\
\end{itemize}
\end{thm}

\begin{proof}

On commence par vérifier que la liste précédente ne contient que des solutions irréductibles. Les propositions 3.1 à 3.4 et la proposition 3.8 montrent que les éléments donnés dans le théorème sont bien les solutions irréductibles de $(E_{7})$ pour $n \leq 4$. On vérifie par un calcul direct que les éléments donnés dans le théorème sont bien des solutions de $(E_{7})$ puis on établit informatiquement la liste de toutes les solutions de $(E_{7})$ pour $n \leq 9$. À partir de celles-ci, on vérifie que les solutions présentes dans le théorème ne peuvent pas s'obtenir comme une somme de deux solutions de taille supérieure à 3. La liste du théorème ne contient donc que des solutions irréductibles.
\\
\\Montrons que les solutions du théorème sont les seules solutions irréductibles de $(E_{7})$. Soit $(\overline{a_{1}},\ldots,\overline{a_{n}})$ une solution de $(E_{7})$ avec $n \geq 9$. Si un des $\overline{a_{i}}$ est égal \`a $\overline{0}$, $\overline{1}$ ou $\overline{-1}$ alors $(\overline{a_{1}},\ldots,\overline{a_{n}})$ est réductible et on suppose donc que cela n'est pas le cas. On peut obtenir la liste de toutes les possibilités pour $(\overline{a_{2}},\ldots,\overline{a_{8}})$ en établissant informatiquement la liste des 7-uplets d'éléments compris dans $\{\overline{2}, \overline{3}, \overline{4}, \overline{5}\}$. Dans cette liste, on élimine toutes les possibilités permettant d'écrire $(\overline{a_{1}},\ldots,\overline{a_{n}})$ comme étant équivalent à la somme d'une solution avec une des solutions irréductibles de la liste ci-dessus. Une fois tout ces éléments retirés il reste : 
\\$(\overline{2},\overline{3},\overline{5},\overline{4},\overline{3},\overline{5},\overline{4}),(\overline{2},\overline{3},\overline{5},\overline{5},\overline{4},\overline{2},\overline{3}),(\overline{2},\overline{5},\overline{3},\overline{4},\overline{5},\overline{3},\overline{4}),(\overline{2},\overline{5},\overline{3},\overline{5},\overline{2},\overline{4},\overline{3}),(\overline{3},\overline{4},\overline{2},\overline{3},\overline{4},\overline{2},\overline{5}),
(\overline{3},\overline{4},\overline{2},\overline{5},\overline{3},\overline{5},\overline{2}),
\\(\overline{3},\overline{2},\overline{4},\overline{3},\overline{2},\overline{4},\overline{5}),(\overline{3},\overline{2},\overline{4},\overline{5},\overline{5},\overline{3},\overline{2}),(\overline{3},\overline{2},\overline{2},\overline{4},\overline{2},\overline{2},\overline{4}),(\overline{4},\overline{3},\overline{5},\overline{4},\overline{3},\overline{5},\overline{2}),(\overline{4},\overline{3},\overline{5},\overline{2},\overline{4},\overline{2},\overline{5}), (\overline{4},\overline{5},\overline{3},\overline{4},\overline{5},\overline{3},\overline{2}),
\\(\overline{4},\overline{5},\overline{3},\overline{2},\overline{2},\overline{4},\overline{5}),(\overline{5},\overline{4},\overline{2},\overline{3},\overline{4},\overline{2},\overline{3}),(\overline{5},\overline{4},\overline{2},\overline{2},\overline{3},\overline{5},\overline{4}),(\overline{5},\overline{3},\overline{2},\overline{2},\overline{4},\overline{2},\overline{2}),(\overline{5},\overline{2},\overline{4},\overline{3},\overline{2},\overline{4},\overline{3}), (\overline{5},\overline{2},\overline{4},\overline{2},\overline{5},\overline{3},\overline{4})$. 
\\
\\Pour chacun de ces 7-uplets on peut considérer les 4 possibilités pour le 7-uplet $(\overline{a_{1}},\ldots,\overline{a_{7}})$ (ou le 7-uplet $(\overline{a_{3}},\overline{a_{1}},\ldots,\overline{a_{9}})$). Chacune de ces possibilités contient un $k$-uplet permettant d'écrire $(\overline{a_{1}},\ldots,\overline{a_{n}})$ comme étant équivalent à une somme d'une solution avec une des solutions irréductibles de la liste ci-dessus. 
\\
\\ Donc, $(\overline{a_{1}},\ldots,\overline{a_{n}})$ est équivalent à la somme d'un $k$-uplet avec un $l$-uplet avec $3 \leq l \leq 9$. En particulier, si $n \geq 10$ alors $k \geq 3$ et $(\overline{a_{1}},\ldots,\overline{a_{n}})$ est réductible. Donc, les solutions irréductibles de $(E_{7})$ sont celles données dans la liste ci-dessus.

\end{proof}

\section{Quelques conjectures et problèmes ouverts}

Tous les cas traités dans la section précédente nous amènent aux deux conjectures suivantes:

\begin{con}

Soit $N \in \mathbb{N}^{*}$, $N \geq 2$. $(E_{N})$ possède un nombre fini de solutions irréductibles.

\end{con}

\begin{con}

Il existe un entier strictement positif $K$ tel que pour tout entier $N$ supérieur à 2 les solutions irréductibles de $(E_{N})$ sont de taille inférieure à $N+K$.

\end{con}

\noindent Comme $\forall n \in \mathbb{N}^{*}$, $(E_{N})$ a un nombre fini de solutions de taille $n$, la conjecture 2 implique la conjecture 1.
\\
\\ \indent Les solutions monomiales minimales sont des solutions particulièrement intéressantes de $(E_{N})$. On sait que si $N$ est premier alors elles sont de taille inférieure ou égale à $N$. Il serait intéressant d'avoir plus d'informations sur leur taille dans les cas $N$ premier et $N$ non premier. Ceci nous amène à formuler le problème suivant :

\begin{pro}

Étudier les tailles des solutions monomiales minimales de $(E_{N})$ dans le cas $N$ premier et dans le cas général.

\end{pro}

En particulier, dans le cas où $N=l^{n}$, avec $l \geq 2$, on a une solution monomiale donnée dans la proposition 3.14. Il serait intéressant d'avoir plus d'informations sur ces solutions.

\begin{pro}

Les solutions données dans la proposition 3.14 sont-elles monomiales minimales ? Si oui, sont-elles irréductibles ?

\end{pro}

Un autre problème ouvert est la généralisation des propositions 4.7 et 4.10. Pour cela on définit la notion de solution compatible avec une triangulation. Une solution de $(E_{N})$ de taille $n$ est dite compatible avec une triangulation s'il existe un découpage d'un polygone convexe à $n$ sommets n'utilisant que des triangles de poids $\overline{1}$ ou $\overline{-1}$ dont elle est la quiddité. Le problème se formule alors de la façon suivante :

\begin{pro}

Soit $N \geq 2$. Caractériser les solutions de $(E_{N})$ compatibles avec une triangulation.

\end{pro}

On peut remarquer que l'on peut généraliser l'argument de la proposition 4.7 de la façon suivante. Soit $(\overline{a_{1}},\ldots,\overline{a_{n}})$ une solution de $(E_{N})$ telle que celle-ci est la quiddité d'un découpage d'un polygone convexe à $n$ sommets n'utilisant que des triangles de poids $\overline{1}$ ou $\overline{-1}$ et des quadrilatères de poids $\overline{0}$. Si $(\overline{a_{1}},\ldots,\overline{a_{n}}) \neq (\overline{0},\ldots,\overline{0})$ alors elle est la quiddité d'un découpage d'un polygone convexe à $n$ sommets n'utilisant que des triangles de poids $\overline{1}$ ou $\overline{-1}$. 
\\
\\En effet, si la décomposition contient un quadrilatère alors elle contient nécessairement un triangle qui partage un côté avec un quadrilatère (sinon celle-ci ne contiendrait que des quadrilatères et aurait donc pour quiddité $(\overline{0},\ldots,\overline{0})$). On procède alors à la transformation ci-dessous (suivant le poids du triangle) :
$$
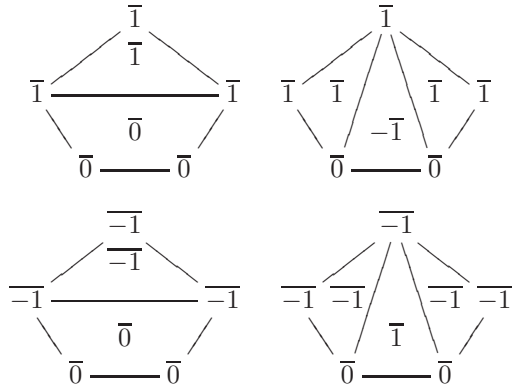
\shorthandoff{; :!?}
\qquad
\xymatrix @!0 @R=0.50cm @C=0.65cm
{
&&\overline{1}\ar@{-}[rrdd]\ar@{-}[lldd]&
\\
&&\overline{1}
\\
\overline{1}\ar@{-}[rdd]\ar@{-}[rrrr]&&&& \overline{1}\ar@{-}[ldd]\\
&&\overline{0}&
\\
&\overline{0}\ar@{-}[rr]&& \overline{0}
}
\quad
\xymatrix @!0 @R=0.50cm @C=0.65cm
{
&&\overline{1}\ar@{-}[rrdd]\ar@{-}[lldd]\ar@{-}[ldddd]\ar@{-}[rdddd]&
\\
\\
\overline{1}\ar@{-}[rdd]&\overline{1}&&\overline{1}& \overline{1}\ar@{-}[ldd]
\\
&&-\overline{1}
\\
&\overline{0}\ar@{-}[rr]&& \overline{0}
}
$$

$$
\shorthandoff{; :!?}
\qquad
\xymatrix @!0 @R=0.50cm @C=0.65cm
{
&&\overline{-1}\ar@{-}[rrdd]\ar@{-}[lldd]&
\\
&&\overline{-1}
\\
\overline{-1}\ar@{-}[rdd]\ar@{-}[rrrr]&&&& \overline{-1}\ar@{-}[ldd]\\
&&\overline{0}&
\\
&\overline{0}\ar@{-}[rr]&& \overline{0}
}
\quad
\xymatrix @!0 @R=0.50cm @C=0.65cm
{
&&\overline{-1}\ar@{-}[rrdd]\ar@{-}[lldd]\ar@{-}[ldddd]\ar@{-}[rdddd]&
\\
\\
\overline{-1}\ar@{-}[rdd]&\overline{-1}&&\overline{-1}& \overline{-1}\ar@{-}[ldd]
\\
&&\overline{1}
\\
&\overline{0}\ar@{-}[rr]&& \overline{0}
}
$$

\noindent On recommence ce procédé tant qu'il reste des quadrilatères. 

\medskip

\noindent {\bf Remerciements}.
Je remercie Valentin Ovsienko et Michael Cuntz pour leurs suggestions et leurs conseils avisés.

\end{document}